\documentclass{amsart}
\usepackage{graphicx} % Required for inserting images
\usepackage[letterpaper]{geometry}

\usepackage{cite}

\usepackage{amsmath,amssymb,amsthm}
\usepackage{mathtools}
\usepackage{multicol}
\usepackage[colorlinks=true, allcolors=blue]{hyperref}

\numberwithin{equation}{section}

\theoremstyle{plain}
\newtheorem{theorem}{Theorem}
\newtheorem{lemma}{Lemma}

\newtheorem{prop}[lemma]{Proposition}
\newtheorem{coro}[lemma]{Corollary}

\newtheorem{assump}{Assumption}

\numberwithin{lemma}{section}

\theoremstyle{definition}
\newtheorem{defn}{Definition}
\newtheorem*{notation}{Notation}

\theoremstyle{remark}
\newtheorem*{remark}{Remark}

\usepackage{tcolorbox}

\usepackage[toc,page]{appendix}
\usepackage{algorithm}
\usepackage[noend]{algpseudocode}

\usepackage{mathrsfs}
\DeclareFontFamily{U}{rsfs}{\skewchar\font127 }
\DeclareFontShape{U}{rsfs}{m}{n}{%
   <-6.5> rsfs5
   <6.5-8> rsfs7
   <8-> rsfs10
}{}

\newcommand{\ii}{\mathrm{i}}
\newcommand{\dd}{\mathrm{d}}

\usepackage{bbm}

\newcommand{\norm}[1]{\lVert #1 \rVert}

\usepackage{todonotes}

\title{Solution Theory of Hamilton-Jacobi-Bellman Equations in Spectral Barron Spaces}
\thanks{This research is supported in part by the National Science Foundation under awards DMS-2309378 and IIS-2403276.}
\date{\today} % change to the submission date
\author{Ye Feng}
\address{Mathematics Department, Duke University}
\email{ye.feng@duke.edu}

\author{Jianfeng Lu}
\address{Department of Mathematics, Department of Physics, Department of Chemistry, Duke University}
\email{jianfeng@math.duke.edu}

\begin{document}

\begin{abstract}
    We study the solution theory of the whole-space static (elliptic) Hamilton-Jacobi-Bellman (HJB) equation in spectral Barron spaces. We prove that under the assumption that the coefficients involved are spectral Barron functions and the discount factor is sufficiently large, there exists a sequence of uniformly bounded spectral Barron functions that converges locally uniformly to the solution. As a consequence, the solution of the HJB equation can be approximated by two-layer neural networks without curse of dimensionality. 
\end{abstract}

\maketitle

\section{Introduction}

Neural networks offer a promising way to tackle numerical solutions to high dimensional partial differential equations (PDEs) by employing a neural network ansatz to the unknown solution and determining parameters in the neural network through optimization (training). This approach has achieved remarkable empirical success over recent years, see e.g., for high dimensional Schr\"odinger equations \cite{hermann2020deep, han2019solving, pfau2020ab}, high dimensional optimal control problems \cite{li2024neural, nusken2021solving,zhou2021actor, han2016deep, nakamura2021adaptive}. For numerical approaches to other high dimensional PDEs based on neural networks, we refer to the reviews \cite{beck2020overview, han2018solving, weinan2021algorithms} and the references therein.   

Alongside these empirical successes, there have been considerable efforts in theoretical analysis, in particular trying to understand the impact of dimension. To overcome the curse of dimensionality, novel solution theory needs to be established for PDEs in high dimension, one of the key issues is to identify function spaces that are appropriate for both approximations using neural networks and also solution theory for high dimensional PDEs. 
The Barron spaces have been often used thanks to its natural connection with two-layer neural network approximation. In the seminal work of Barron \cite{barron1993universal}, it was established that any function $f$ satisfying 
\begin{align}\label{eq:barron_const}
    C_f := \int_{\mathbb{R}^d} |\hat{f}(\xi)| |\xi| \dd \xi < \infty  
\end{align}
can be approximated efficiently in $L^2$-norm by a two-layer neural network $f_n$ of the form 
\begin{align}
    f_n(x) = \sum_{j=1}^n a_j \sigma(w_j \cdot x + b_j), \nonumber
\end{align}
where the number of neurons $n$ is independent of the dimension $d$. Following \cite{barron1993universal}, several function spaces have been proposed to characterize appropriate function classes for neural network approximation. These include the spectral Barron spaces that generalize the Barron's construction (\ref{eq:barron_const}) to some weighted $L^1$-norm of the Fourier transform $\hat{f}$, and the Barron spaces that are defined via the integral representation of $f$ based on some probability distribution on the parameter space. We refer to \cite{wojtowytsch2022representation, ma2022barron} for a comprehensive overview of the function spaces along the direction of Barron. 

Previous works have established solution theory of high dimensional PDEs in Barron space. So far, most works have been restricted to linear elliptic equations, see \cite{lu2021priori, lu2022priori, chen2021representation, chen2023regularity}. While linear PDEs are widely used in applications, one natural question is whether we can extend our understanding of high dimensional PDEs to nonlinear cases, such as the Hamilton-Jacobi-Bellman type equations arising from control problems. 
As far as we are aware, the only work devoted to the solution theory to nonlinear PDEs in Barron space is \cite{marwah2023neural}, where nonlinear elliptic PDEs arising from certain type of calculus of variation was considered. 

%\textit{If the coefficients of a PDE are in (spectral) Barron space, can the solution to the PDE also be in (spectral) Barron spaces, or be approximated by (spectral) Barron functions?}

%A positive answer\jlnotes{try to merge this paragraph with above} to this question usually leads to the result that the solution to the PDE can be approximated by two-layer neural networks without curse of dimensionality. To our knowledge, the answer has been shown affirmative for some linear high-dimensional PDEs. For example, \cite{chen2023regularity} shows that the solution to the $d$-dimensional static Schr\"odinger equation lies in the spectral Barron space if potential function and the source function are spectral Barron. The work of \cite{chen2021representation} considers $d$-dimensional linear second-order elliptic PDEs and shows that if the coefficients and the source term of the elliptic PDE lie in Barron spaces, then the solution of the PDE can be approximated by Barron functions with respect to the $H^1$-norm. 

The main goal of this paper is to study high dimensional static Hamilton-Jacobi-Bellman (HJB) equations in the whole space: %\jlnotes{I think here we can be a bit more generic, for example, perhaps we don't really need to specify that the running cost depends quadratically on $u$? We can then give more details of the equations etc when we later state the results?}
\begin{align}\label{eq:HJB_intro}
    - \gamma V(x) + \min_{u \in \mathbb{R}^m} \{\mathcal{L}_u V(x) + l(x,u)\} = 0, \hspace{1em} x \in \mathbb{R}^d,
\end{align}
where $\gamma > 0$ is the discount factor, the infinitesimal generator $\mathcal{L}_u = (f + gu) \cdot \nabla + \Delta$ with $f$ and $g$ being vector/matrix-valued functions of suitable sizes, $l: \mathbb{R}^d \times \mathbb{R}^m \to \mathbb{R}$ is the cost function. A solution of the HJB equation (\ref{eq:HJB_intro}) is a pair of functions $(u, V)$, often referred to as an optimal solution, that satisfies equation (\ref{eq:HJB_intro}) with $u(x) = \mathrm{argmin}_{u \in \mathbb{R}^m} \mathcal{L}_u V(x) +l(x,u)$ for each $x \in \mathbb{R}^d$.

Our main result shows that given the coefficients $f, g$ and $l$ in spectral Barron spaces, if the cost function $l$ is quadratic in $u$ and the discount $\gamma$ is large enough, then there exists a sequence of spectral Barron functions that converges to the optimal solution of (\ref{eq:HJB_intro}) locally uniformly on $\mathbb{R}^d$. Moreover, the sequence of spectral Barron functions is uniformly bounded in the spectral Barron norm, and, as a result, the optimal solution can be approximated by two-layer neural networks without curse of dimensionality. Our result hence provides theoretical support to using neural network approximations for high-dimensional elliptic HJB equations. 

\subsection*{Related works}

Solution theory of PDEs in spectral Barron spaces has been studied in several recent works. In \cite{lu2021priori}, the authors establish regularity results for the Poisson equation and the Schr\"odinger equation on the bounded domain $[0,1]^d$ with homogeneous Neumann boundary condition in a discrete version of spectral Barron space via Fourier cosine expansion, and later similar results are extended to Schr\"odinger eigenvalue problems in \cite{lu2022priori}. The work of \cite{chen2023regularity} proves the regularity of the static Schr\"odinger equation $-\Delta u + Vu = f$ in spectral Barron spaces. In \cite{marwah2023neural}, the authors focus on a class of PDEs known as nonlinear elliptic variational PDEs and establishes the approximation of the solution in spectral Barron space. Recently, it is proved in \cite{yserentant2025regularity} that the solutions of the $N$-electron Schr\"odinger equation with Coulomb singularity for eigenvalues below the essential spectrum lie in the spectral Barron spaces $\mathcal{B}^{s}(\mathbb{R}^{3N})$ for $s < 1$.

Besides the spectral Barron spaces, solution theory of PDEs is also studied in the Barron space based on integral representation. For example, \cite{chen2021representation} approximates the solution of a class of whole-space linear second-order elliptic PDEs in Barron space. The work of \cite{weinan2022some} proves regularity results of the screened Poisson equation $- \Delta u + \lambda^2 u = f$, the heat equation, and a viscous Hamilton-Jacobi equation $u_t - \Delta u + |\nabla u|^2 = 0$ in the Barron space based on integral representation. It also provides some counter-examples when singularity is presented. 

Efforts have also been made in investigating the complexity of approximating PDE solutions via neural networks outside the framework of (spectral) Barron spaces. For example, \cite{Marwah2021ParametricCB} obtains a deep neural network complexity estimate for elliptic PDEs with a homogeneous Dirichlet boundary condition through a proof technique that simulates gradient descent. The study in \cite{grohs2021deep} uses a direct constructive proof to show that the solution to the parabolic HJB equation can be approximated by neural networks without curse of dimensionality. 

\section{Main Results}\label{sec:main_results}

Our focus is to establish an approximation of the solution of the HJB equation (\ref{eq:HJB_intro}) using spectral Barron functions. First, we introduce some notation used in the paper and recall the spectral Barron norm.

\begin{notation}
    We use $|x|$ for the Euclidean norm of a vector $x \in \mathbb{R}^d$ and $B_r(x) = \{y \in \mathbb{R}^d \mid |y - x| < r\}$ for the open ball in $\mathbb{R}^d$ centered at $x$ with radius $r$. For $f \in L^1(\mathbb{R}^d)$, denote its Fourier transform by $\hat{f}$, given by \begin{align}
    \hat{f}(\xi) = \dfrac{1 }{(2\pi)^d} \int_{\mathbb{R}^d} f(x) e^{-\ii x \cdot \xi} \dd x. \nonumber
\end{align}
One can extend the definition of Fourier transform for $f \in \mathcal{S}'(\mathbb{R}^d)$. If $\hat{f} \in L^1(\mathbb{R}^n)$, its inverse Fourier transform is given by 
\begin{align}
    f(x) = \int_{\mathbb{R}^d} \hat{f}(\xi) e^{\ii x \cdot \xi} \dd \xi. \nonumber
\end{align}
\end{notation}

We define the spectral Barron space following the definitions used in \cite{xu2020finite, siegel2022high, siegel2023characterization, siegel2024sharp}. Similar (equivalent) definitions of spectral Barron norms were also used in \cite{chen2023regularity, yserentant2025regularity, lu2022priori, marwah2023neural, weinan2022some, wojtowytsch2022representation}. 

\begin{defn}
    Let $s \in \mathbb{R}$ and $f \in \mathcal{S}'(\mathbb{R}^d)$. The spectral Barron norm of $f$ is defined as 
    \begin{align}
        \Vert f \Vert_{\mathcal{B}^s(\mathbb{R}^d)} := \int_{\mathbb{R}^d} |\hat{f}(\xi)| \left(1 + |\xi|\right)^{s} \dd \xi. \nonumber
    \end{align}
    The spectral Barron space (of order $s$) is defined as 
    \begin{align}
        \mathcal{B}^s(\mathbb{R}^d) := \left\{ f \in \mathcal{S}'(\mathbb{R}^d) \mid \Vert f \Vert_{\mathcal{B}^s(\mathbb{R}^d)} < \infty\right\}. \nonumber
    \end{align}
\end{defn}

The reason that spectral Barron spaces are of particular interest for high-dimensional problems is that a spectral Barron function possesses certain smoothness related to the order of the spectral Barron norm and can be efficiently approximated in $H^k$ norm by a two-layer neural network on any bounded domain without the curse of dimensionality; see Proposition \ref{prop:nn_apprx_spec_barron} in Section \ref{sec:proofs} below. The completeness of the spectral Barron space comes from the fact that the spectral Barron norm is a weighted $L^1$ norm of the Fourier transform. Therefore, the spectral Barron space $\mathcal{B}^{s}(\mathbb{R}^d)$ is a Banach space. 

Since we frequently deal with vector/matrix-valued functions throughout this paper, we introduce the following convention to ease our notation.

\begin{notation}
    For any matrix-valued function $f: \mathbb{R}^d \to \mathbb{R}^{m \times n}$, we say $f = (f_{ij}) \in \mathcal{B}^{s}(\mathbb{R}^d)$ if $f_{ij} \in \mathcal{B}^{s}(\mathbb{R}^d)$ for every $1 \leq i \leq m$ and $1 \leq j \leq n$ and denote
    \begin{align}
        \Vert f \Vert_{\mathcal{B}^{s}(\mathbb{R}^d)} := \sum_{i=1}^{m} \sum_{j=1 }^{n} \Vert f_{ij} \Vert_{\mathcal{B}^{s}(\mathbb{R}^d)}. \nonumber
    \end{align}
\end{notation}

In this paper, we consider the cost function to be quadratic in $u$, that is, $l(x,u) = \ell(x) + \norm{u}_R^2$ where $\ell: \mathbb{R}^d \to [0, \infty)$ is the running cost and $\Vert u \Vert_R^2 = \langle u, u \rangle_R:= u^T R u$ is the control cost defined by some positive definite matrix $R \in \mathbb{R}^{m \times m}$. In this case, the HJB equation (\ref{eq:HJB_intro}) becomes
\begin{align}\label{eq:HJB_qd}
    - \gamma V(x) + \min_{u \in \mathbb{R}^m} \{\mathcal{L}_u V(x) + \ell(x) + \Vert u \Vert_R^2\} = 0, \hspace{1em} x \in \mathbb{R}^d.
\end{align}
We remark that it is straightforward to extend the results to the case where the matrix $R = R(x)$ depends on $x$ (with the assumption that $R(x)$ is uniformly positive definite and lies in spectral Barron space), while we stay with constant $R$ for simplicity.

To state our main result, we make the assumption that the coefficient functions $f$, $g$, $\ell$ lie in spectral Barron spaces. In particular, we assume $g$ to be non-zero, otherwise the solution of equation (\ref{eq:HJB_qd}) is trivial. The order of the spectral Barron norm is specified in the statement of the subsequent theorems. 
\begin{assump}\label{asp:coef}
    Assume that the coefficients $f, g, \ell\in \mathcal{B}^{s}(\mathbb{R}^d)$ such that $\Vert g \Vert_{\mathcal{B}^{s}(\mathbb{R}^d)} > 0$. 
\end{assump}

Our main theorem is stated as follows. 

\begin{theorem}\label{thm:main_uni_conv}
    Suppose that Assumption \ref{asp:coef} holds for $s \geq 2$. Then there exists a positive constant $C(f, g, \ell, R)$ depending on $f$, $g$, $\ell$ and $R$ for which if the discount $\gamma \geq C(f, g, \ell, R)$, then there exists a sequence $\{u^{(i)}\}$ in $\mathcal{B}^{s}(\mathbb{R}^d)$ converging to $\bar{u}$ in $C^1_{\mathrm{loc}}(\mathbb{R}^d)$ and a sequence $\{V^{(i)}\}$ in $\mathcal{B}^{s+1}(\mathbb{R}^d)$ converging to $\bar{V}$ in $C^2_{\mathrm{loc}}(\mathbb{R}^d)$, where $(\bar{u}, \bar{V})$ is an optimal solution of the HJB equation (\ref{eq:HJB_qd}). 
\end{theorem}

Theorem \ref{thm:main_uni_conv} implies the existence of a solution to (\ref{eq:HJB_qd}), which can be approximated locally uniformly by spectral Barron functions. The next theorem shows that the solution given by Theorem~\ref{thm:main_uni_conv} can be approximated by two-layer neural networks on any compact domain without curse of dimensionality. 

\begin{theorem}\label{thm:main_nn_appx}
    Under the same assumptions of Theorem \ref{thm:main_uni_conv}, let $\{u^{(i)}\}$ and $\{V^{(i)}\}$ be the sequences of spectral Barron functions and $(\bar{u}, \bar{V})$ be the optimal solution of the HJB equation (\ref{eq:HJB_qd}) given by Theorem \ref{thm:main_uni_conv}. Then for any compact domain $K \subset \mathbb{R}^d$ and for any $n \in \mathbb{N}$, there exist cosine-activated two-layer neural networks with $n$ hidden neurons
    \begin{align}
        u_n(x) = \dfrac{1}{n} \sum_{j=1}^{n} a_j \cos(w_j \cdot x + b_j), \quad V_n(x) = \dfrac{1}{n} \sum_{j=1}^{n} c_j \cos(v_j \cdot x + d_j), \nonumber
    \end{align}
    where $a_j, b_j, c_j, d_j \in \mathbb{R}$ and $w_j, v_j \in \mathbb{R}^d$ for each $j$, such that 
    \begin{align}
        \Vert u_n - \bar{u} \Vert_{H^1(K)} \leq \sqrt{|K|} \left(1 + \sup_{i \in \mathbb{N}} \Vert u^{(i)} \Vert_{\mathcal{B}^{1}(\mathbb{R}^d)}  \right) n^{-1/2}, \nonumber
    \end{align}
    and 
    \begin{align}
        \Vert V_n - \bar{V} \Vert_{H^2(K)} \leq \sqrt{|K|} \left(1 + \sup_{i \in \mathbb{N}} \Vert V^{(i)} \Vert_{\mathcal{B}^{2}(\mathbb{R}^d)}  \right) n^{-1/2}, \nonumber
    \end{align}
    where $|K|$ denotes the Lebesgue measure of $K$. 
\end{theorem}

\begin{remark}
    In Theorem \ref{thm:main_nn_appx}, we focus on neural networks with cosine activation functions, as these are most directly connected to the spectral Barron space, which is defined via the Fourier transform. As shown in \cite{barron1993universal} and \cite{siegel2020approximation}, the spectral Barron space also characterizes functions that can be efficiently approximated by neural networks with more general activation functions, such as sigmoidal, ReLU, or SoftPlus. Consequently, the results of Theorem \ref{thm:main_nn_appx} can be extended to a broader class of activation functions, though we will not go into details.
\end{remark} 

\section{Proofs}\label{sec:proofs}

This section is devoted to the proof of Theorem \ref{thm:main_uni_conv} and Theorem \ref{thm:main_nn_appx}. 

\subsection{Preliminaries} 

In this subsection, we present some preliminary results on the properties of the spectral Barron space. First, we state the approximation theorem for spectral Barron functions. The proof is given in Appendix \ref{app:proofs_conv}.  

\begin{prop}\label{prop:nn_apprx_spec_barron}
    Let $k \in \mathbb{N}$. If $f \in \mathcal{B}^{k}(\mathbb{R}^d)$, then $f \in C^k(\mathbb{R}^d)$ and for any compact domain $K \subset \mathbb{R}^d$ and any $n \in \mathbb{N}$, there exists a cosine-activated two-layer neural network with $n$ hidden neurons 
    \begin{align}
        f_n(x) = \dfrac{1 }{n} \sum_{j=1}^{n} a_j \cos\left( w_j \cdot x + b_j \right) \nonumber
    \end{align}
    where $a_j, b_j \in \mathbb{R}$ and $w_j \in \mathbb{R}^d$ for each $j$, such that 
    \begin{align}
        \Vert f_n - f \Vert_{H^k(K)} \leq \sqrt{|K|} \Vert f \Vert_{\mathcal{B}^{k}(\mathbb{R}^d)} n^{-1/2}, \nonumber
    \end{align}
    where $|K|$ is the Lebesgue measure of $K$. 
\end{prop}

The following embedding properties are immediate from the definition of spectral Barron norm. 

\begin{prop}\label{prop:barron_embed}
   The following embeddings hold. 
    \begin{itemize}
        \item[(1)] $\mathcal{B}^{s}(\mathbb{R}^d) \hookrightarrow L^{\infty}(\mathbb{R}^d)$ with $\Vert f \Vert_{L^{\infty}(\mathbb{R}^d)} \leq \Vert f \Vert_{\mathcal{B}^{s}(\mathbb{R}^d)}$ for any $f \in \mathcal{B}^{s}(\mathbb{R}^d)$ if $s \geq 0$;
        \item[(2)] $\mathcal{B}^{s'}(\mathbb{R}^d) \hookrightarrow \mathcal{B}^{s}(\mathbb{R}^d)$ with $\Vert f \Vert_{\mathcal{B}^{s}(\mathbb{R}^d)} \leq \Vert f \Vert_{\mathcal{B}^{s'}(\mathbb{R}^d)}$ for any $f \in \mathcal{B}^{s'}(\mathbb{R}^d)$ if $s, s' \in \mathbb{R}$ satisfy $s' \geq s$.
    \end{itemize}
\end{prop}

The proposition below indicates that spectral Barron spaces are closed under product and differentiation of functions. The proof is given in Appendix \ref{app:proofs_conv}. We also remark that the differential operator $\partial_{i}$ is well-defined on the spectral Barron space $\mathcal{B}^{s+1}(\mathbb{R}^d)$ when $s \geq 0$ as a result of the regularity of spectral Barron functions shown in Proposition \ref{prop:nn_apprx_spec_barron}. 

\begin{prop}\label{prop:barron_alg}
    Let $s \geq 0$. Then 
    \begin{itemize}
        \item[(1)] For any $f, g \in \mathcal{B}^s(\mathbb{R}^d)$, $fg \in \mathcal{B}^{s}(\mathbb{R}^d)$ with 
        \begin{align}
            \Vert fg \Vert_{\mathcal{B}^s(\mathbb{R}^d)}  \leq \Vert f \Vert_{\mathcal{B}^s(\mathbb{R}^d)} \Vert g \Vert_{\mathcal{B}^s(\mathbb{R}^d)}. \nonumber
        \end{align}
        \item[(2)] For any $1 \leq i\leq d$, the differential operator $\partial_{i} : \mathcal{B}^{s+1}(\mathbb{R}^d) \to \mathcal{B}^{s}(\mathbb{R}^d)$ is a bounded linear operator with
        \begin{align}
            \Vert \partial_{i} f \Vert_{\mathcal{B}^{s}(\mathbb{R}^d)} \leq \Vert f \Vert_{\mathcal{B}^{s+1}(\mathbb{R}^d)}. \nonumber
        \end{align}
    \end{itemize}
\end{prop}

The following estimates of the bound of $(\gamma I - \Delta)^{-1}$ in spectral Barron spaces will be useful in the subsequent analysis. 

\begin{prop}\label{prop:bdd_gamma_delta}
    Let $s \geq 0$. If $ \gamma > 0$, then for any $f \in \mathcal{B}^{s}(\mathbb{R}^d)$, 
    \begin{align}
        \Vert (\gamma I - \Delta)^{-1} f \Vert_{\mathcal{B}^{s}(\mathbb{R}^d)}  & \leq \gamma^{-1} \Vert f \Vert_{\mathcal{B}^{s}(\mathbb{R}^d)}, \nonumber \\
        \Vert (\gamma I - \Delta)^{-1} f \Vert_{\mathcal{B}^{s+1}(\mathbb{R}^d)} & \leq \dfrac{1 }{2(\sqrt{1+\gamma} -1)} \Vert f \Vert_{\mathcal{B}^{s}(\mathbb{R}^d)}, \nonumber \\
        \Vert (\gamma I - \Delta)^{-1} f \Vert_{\mathcal{B}^{s+2}(\mathbb{R}^d)} & \leq \left(1+ \dfrac{1}{\gamma}\right) \Vert f \Vert_{\mathcal{B}^{s}(\mathbb{R}^d)}. \nonumber
    \end{align}
\end{prop}

\begin{proof}
    One can easily compute that 
    \begin{align}
        \Vert (\gamma I - \Delta)^{-1} f \Vert_{\mathcal{B}^{s}(\mathbb{R}^d)} & = \int_{\mathbb{R}^d} \dfrac{1 }{\gamma + |\xi|^2} |\hat{f}(\xi)| \left(1 + |\xi|\right)^{s} \dd \xi \nonumber \\
        & \leq \gamma^{-1} \int_{\mathbb{R}^d} |\hat{f}(\xi)| \left(1 + |\xi|\right)^{s} \dd \xi = \gamma^{-1} \Vert f \Vert_{\mathcal{B}^{s}(\mathbb{R}^d)}. \nonumber
    \end{align}
    Simple calculation shows that for all $\xi \in \mathbb{R}^d$, 
    \begin{align}
        \dfrac{1+|\xi|}{\gamma + |\xi|^2} \leq \dfrac{1 }{2(\sqrt{1+\gamma} -1)}, \hspace{1em} \dfrac{(1+|\xi|)^2}{\gamma + |\xi|^2} \leq 1 + \dfrac{1}{\gamma}. \nonumber
    \end{align} 
    Therefore,  
    \begin{align}
        \Vert (\gamma I - \Delta)^{-1} f \Vert_{\mathcal{B}^{s+1}(\mathbb{R}^d)} & = \int_{\mathbb{R}^d} \dfrac{1 + |\xi|}{\gamma + |\xi|^2} |\hat{f}(\xi)| \left(1 + |\xi|\right)^{s} \dd \xi \nonumber \\
        & \leq \dfrac{1 }{2(\sqrt{1+\gamma} -1)}\int_{\mathbb{R}^d} |\hat{f}(\xi)| \left(1 + |\xi|\right)^{s} \dd \xi = \dfrac{1 }{2(\sqrt{1+\gamma} -1)} \Vert f \Vert_{\mathcal{B}^{s}(\mathbb{R}^d)}, \nonumber
    \end{align}
    and
    \begin{align}
        \Vert (\gamma I - \Delta)^{-1} f \Vert_{\mathcal{B}^{s+2}(\mathbb{R}^d)} & = \int_{\mathbb{R}^d} \dfrac{(1 + |\xi|)^2}{\gamma + |\xi|^2} |\hat{f}(\xi)| \left(1 + |\xi|\right)^{s} \dd \xi \nonumber \\
        & \leq \left(1+ \dfrac{1}{\gamma}\right) \int_{\mathbb{R}^d} |\hat{f}(\xi)| \left(1 + |\xi|\right)^{s} \dd \xi = \left(1+ \dfrac{1}{\gamma}\right) \Vert f \Vert_{\mathcal{B}^{s}(\mathbb{R}^d)}. \nonumber \qedhere
    \end{align}
\end{proof}

\subsection{Control formulation of HJB equations}

In this subsection, we reformulate the HJB equation (\ref{eq:HJB_intro}) into an optimal control problem. We assume that the coefficients $f : \mathbb{R}^d \to \mathbb{R}^d$ and $g: \mathbb{R}^d \to \mathbb{R}^{d \times m}$ are Lipschitz continuous and globally bounded on $\mathbb{R}^d$. We make the above regularity and boundedness assumptions on $f$ and $g$ here for the purpose of introducing the optimal control problem. In proving our main results, assuming $f$, $g$ are spectral Barron functions of order $s \geq 1$ automatically satisfies these conditions. 

Let $(\Omega, \mathcal{F}, \{\mathcal{F}_t\}_{t \geq 0}, \mathbb{P})$ be a filtered probability space such that $(\Omega, \mathcal{F}, \mathbb{P})$ is complete, $\mathcal{F}_0$ contains all the $\mathbb{P}$-null sets in $\mathcal{F}$, and $\{ \mathcal{F}_t\}_{t \geq 0}$ is right continuous. Let $W_t$ be an $n$-dimensional standard Brownian motion defined on $(\Omega, \mathcal{F}, \{\mathcal{F}_t\}_{t \geq 0}, \mathbb{P})$. Consider the following stochastic differential equation (SDE)
\begin{align}\label{eq:state_dyn}
    dX_t = \left( f(X_t) + g(X_t) u_t \right) dt + \sqrt{2} dW_t, \hspace{1em} X_0 = x \in \mathbb{R}^d, 
\end{align}
where $u(\cdot): [0,\infty) \times \Omega \to \mathbb{R}^m$ is $\{\mathcal{F}_t\}_{t \geq 0}$-progressively measurable. We then consider an infinite horizon optimal control problem to minimize the following cost functional: 
\begin{align}
   J(x, u(\cdot)) := \mathbb{E}\left[ \int_{0}^{\infty} e^{-\gamma t} \left(\ell(X_t) + \Vert u_t \Vert_R^2 \right) dt \mid X_0 = x\right]. 
\end{align}
The control $u(\cdot)$ is chosen over the set 
\begin{align}
    \mathcal{U}:= \{u: [0, \infty) \times \Omega \to \mathbb{R}^m \mid u(\cdot) \text{ is } \{\mathcal{F}_t\}_{t \geq 0}\text{-progressively measurable}\}. \nonumber
\end{align}
The optimal value function is defined as 
\begin{align}\label{eq:opt_ctrl_prb}
    V(x) := \min_{u(\cdot) \in \mathcal{U}} J(x, u(\cdot)), \hspace{1em}  x \in \mathbb{R}^d. 
\end{align}
According to standard stochastic control theory \cite{yong2012stochastic}, if the optimal value function $V$ exists and $V \in C^2(\mathbb{R}^d)$, then $V$ satisfies the HJB equation (\ref{eq:HJB_qd}). %Throughout the paper, we assume that the optimal value function $V$ exists and $V \in C^2(\mathbb{R}^d)$. In this case, according to standard stochastic control theory \cite{yong2012stochastic}, $V$ satisfies the HJB equation (\ref{eq:HJB_qd}). 
The optimal feedback control $u^*$ is given by 
\begin{align}
    u^*(x) & = \mathrm{argmin}_{u \in \mathbb{R}^m}\left\{ \mathcal{L}_u V(x) + \ell(x) + \Vert u \Vert_R^2 \right\} = - \dfrac{1 }{2 } R^{-1} g(x)^T \nabla V(x), \hspace{1em} x \in \mathbb{R}^d. \nonumber
\end{align}
Therefore, an equivalent form of the HJB equation (\ref{eq:HJB_qd}) is 
\begin{align}\label{eq:HJB_equiv}
    - \gamma V + \nabla V \cdot \left( f+g \left( - \dfrac{1}{2} R^{-1} g^T \nabla V \right)\right) + \Delta V + l(x) + \left\lVert - \dfrac{1}{2} R^{-1} g^T \nabla V \right\rVert_R^2 = 0. 
\end{align}

\begin{defn}\label{def:adm_fd}
    A function $u: \mathbb{R}^d \to \mathbb{R}^m$ is called an admissible feedback control if (i) $u$ is continuous and (ii) $\forall x \in \mathbb{R}^d$, there exists a unique (strong) solution $X_t^u$  to the SDE 
    \begin{align}
        dX_t = \left( f(X_t) + g(X_t) u(X_t) \right) dt + \sqrt{2} dW_t, \hspace{1em} X_0 = x. \nonumber
    \end{align}
    such that $\mathbb{E} \left[\int_{0 }^{\infty } e^{-\gamma t} \left(\ell(X_t^u) + \Vert u(X_t^u) \Vert_R^2 \right)\dd t \right] < \infty$. 
\end{defn}

The control formulation helps us represent the solution of the linearization of the HJB equation (\ref{eq:HJB_qd}) in terms of probabilistic expectation that involves the state trajectory given by (\ref{eq:state_dyn}). As we will see in the next subsection, this representation will allow us to establish the monotonicity of the sequence that we construct to approximate the solution of (\ref{eq:HJB_qd}), which is crucial to the convergence of the approximating sequence. The definition of admissibility specifies the set of feedback controls under which such a desired probabilistic representation is possible.

\subsection{Policy iteration}\label{sec:conv_ply_itr}

In this subsection, we use policy iteration to construct sequences of spectral Barron functions to approximate the solution to the HJB equation (\ref{eq:HJB_qd}). The policy iteration is an algorithm to find the optimal solution by iteratively solving the linearized HJB equation and improving the controls; see Algorithm \ref{alg:HJB_2nd_ord}. 

\begin{algorithm}[h]
    \caption{Policy iteration for the elliptic HJB equation (\ref{eq:HJB_qd})}\label{alg:HJB_2nd_ord}
    \begin{algorithmic}
    \State Choose an initial admissible control $u^{(0)}: \mathbb{R}^d \to \mathbb{R}^m$. 
    \State For each $i \in \mathbb{N}$, if there exists $V^{(i)} \in C^2(\mathbb{R}^d)$ that solves the equation
    \begin{align}\label{alg:1_solve_ghjb_2nd}
        -\gamma V^{(i)}(x) + \mathcal{L}_{u^{(i)}} V^{(i)}(x) + \ell(x) + \Vert u^{(i)}(x) \Vert_R^2 = 0,  \hspace{1em} x\in \mathbb{R}^d,
    \end{align}
    update the control by 
    \begin{align}\label{alg:2_update_ctrl_2nd}
        u^{(i+1)}(x) := - \dfrac{1 }{2} R^{-1} g(x)^{T} \nabla V^{(i)}(x), \hspace{1em} x \in \mathbb{R}^d.
    \end{align}
    \end{algorithmic}
\end{algorithm}

We remark that in presenting the algorithm, there is no guarantee that a classical solution to the linearized equation (\ref{alg:1_solve_ghjb_2nd}) always exists. We settle this problem by assuming the coefficient functions $f$, $g$, $\ell$ lie in spectral Barron spaces and address the well-posedness of the linearized HJB equation by showing that if the control is a spectral Barron function, then there exists a unique solution to the linearized equation (\ref{alg:1_solve_ghjb_2nd}) in the spectral Barron space, and by the regularity of spectral Barron functions in Proposition \ref{prop:nn_apprx_spec_barron}, there exists a unique classical solution to (\ref{alg:1_solve_ghjb_2nd}). 

We first prove a regularity result of linear second-order elliptic PDEs in spectral Barron spaces.  

\begin{theorem}\label{thm:reg}
    For any $s \geq 1$, $\gamma > 0$, $\mu: \mathbb{R}^d \to \mathbb{R}^d \in \mathcal{B}^{s}(\mathbb{R}^d)$ and $v_f \in \mathcal{B}^{s}(\mathbb{R}^d)$, there exists a unique solution $v^* \in \mathcal{B}^{s+1}(\mathbb{R}^d)$ to the equation
    \begin{align}\label{eq:elp_pde}
        \gamma v - \Delta v + \mu \cdot \nabla v = v_f \text{ in } \mathbb{R}^d.
    \end{align}
    Further, $v^*$ satisfies
    \begin{align}
        \Vert v^* \Vert_{\mathcal{B}^{s+1}(\mathbb{R}^d)} \leq C \Vert v_f\Vert_{\mathcal{B}^s(\mathbb{R}^d)}, \nonumber
    \end{align}
    where $C$ is a constant depending only on $d$, $\gamma$ and $\mu$. 
\end{theorem}

The proof of Theorem \ref{thm:reg} is given in Appendix \ref{app:reg_elliptic}. A similar regularity theorem of the static Schr\"odinger equation is given in \cite{chen2023regularity}. The main difference between the elliptic PDE (\ref{eq:elp_pde}) discussed here and the one in \cite{chen2023regularity} is that the equation (\ref{eq:elp_pde}) includes a drift term $\mu: \mathbb{R}^d \to \mathbb{R}^d$. Without the drift term $\mu$, one can obtain $\Vert v^* \Vert_{\mathcal{B}^{s+2}(\mathbb{R}^d)} \leq C \Vert v_f \Vert_{\mathcal{B}^{s}(\mathbb{R}^d)}$ following the proof of \cite{chen2023regularity}. However, the difference is insignificant to our problem since, roughly speaking, the source term shares the same order of the control function $u$ while the solution corresponds to the value function $V$, and it is natural from (\ref{alg:2_update_ctrl_2nd}) that they differ in order by 1 in terms of spectral Barron norms. In Theorem \ref{thm:reg}, we specify $s \geq 1$ to ensure that the solution to (\ref{eq:elp_pde}) is classical. However, when considering $s \geq 0$, under the same assumptions, one can prove in exactly the same way that there exists a unique solution $v^* \in \mathcal{B}^{s+1}(\mathbb{R}^d)$ to the equation $(I + \mathcal{T}) v = (\gamma I - \Delta)^{-1} v_f$ in $\mathbb{R}^d$, where $\mathcal{T} = (\gamma I - \Delta)^{-1} \mu \cdot \nabla$. In this case, the Laplacian of $v$ does not necessarily exist in the classical sense. 

We present the following two lemmas, which are evident from the algebras of spectral Barron spaces. The proofs can be found in Appendix \ref{app:proofs_conv}. 

\begin{lemma}\label{lmm:constant_c_R1}
    Let $s \geq 0$. For any control $u \in \mathcal{B}^{s}(\mathbb{R}^d)$, there exists a constant $C_{R,1} > 0$ such that 
\begin{align}
    \Vert \langle u, u\rangle_R \Vert_{\mathcal{B}^{s}(\mathbb{R}^d)} \leq C_{R,1} \Vert u \Vert_{\mathcal{B}^{s}(\mathbb{R}^d)}^2. \nonumber
\end{align}
\end{lemma}

\begin{lemma}\label{lmm:constant_c_R2}
    Let $s \geq 0$ and $g \in \mathcal{B}^{s}(\mathbb{R}^d)$. For any value function $V \in \mathcal{B}^{s+1}(\mathbb{R}^d)$, define 
\begin{align}
    u := - \dfrac{1 }{2} R^{-1} g^T \nabla V, \nonumber
\end{align}
then there exists $C_{R,2} > 0$ such that 
\begin{align}
    \Vert u \Vert_{\mathcal{B}^s(\mathbb{R}^d)} \leq C_{R,2} \Vert g \Vert_{\mathcal{B}^s(\mathbb{R}^d)} \Vert V \Vert_{\mathcal{B}^{s+1}(\mathbb{R}^d)}. \nonumber
\end{align}
\end{lemma}

The following lemma shows that any spectral Barron control $u$ is admissible and the linearized HJB equation 
\begin{align}\label{eq:lin_HJB}
        -\gamma V_u(x) + \mathcal{L}_{u} V_u(x) + \ell(x) + \Vert u(x) \Vert_R^2 = 0, \hspace{1em} x \in \mathbb{R}^d,
    \end{align}
admits a unique spectral Barron solution $V_u$.

\begin{lemma}\label{lmm:barron_adm}
    Suppose that Assumption \ref{asp:coef} holds for $s \geq 1$ and $u :\mathbb{R}^d \to \mathbb{R}^m \in \mathcal{B}^{s}(\mathbb{R}^d)$, then $u$ is an admissible feedback control and there exists a unique solution $V_u \in \mathcal{B}^{s+1}(\mathbb{R}^d)$ to the linearized HJB equation (\ref{eq:lin_HJB}). 
\end{lemma}

\begin{proof}
    By Assumption \ref{asp:coef}, Proposition \ref{prop:barron_alg} (1) and Lemma \ref{lmm:constant_c_R1}, we have $f +g u \in \mathcal{B}^{s}(\mathbb{R}^d)$ and $\ell + \Vert u \Vert_R^2 \in \mathcal{B}^{s}(\mathbb{R}^d)$. Theorem \ref{thm:reg} implies that there exists a unique solution $V_u \in \mathcal{B}^{s+1}(\mathbb{R}^d)$ to the equation (\ref{eq:lin_HJB}). By Proposition \ref{prop:nn_apprx_spec_barron}, $V_u \in C^2(\mathbb{R}^d)$, thus $V_u$ is a classical solution of (\ref{eq:lin_HJB}).

    We now show that $u$ is admissible. Since $f, g, u \in \mathcal{B}^{s}(\mathbb{R}^d)$ for $s \geq 1$, by Propositions \ref{prop:nn_apprx_spec_barron} and \ref{prop:barron_embed}, $f, g, u \in C^1(\mathbb{R}^d)$ and are globally bounded. Moreover, $f, g, u$ are Lipschitz continuous on $\mathbb{R}^d$ as their gradients lie in $\mathcal{B}^{s}(\mathbb{R}^d)$. By standard SDE theory (for instance, Theorem 5.4 of \cite{klebaner2012introduction}), there exists a unique (strong) solution $X_t^{u}$ of the SDE
    \begin{align}
        dX_t^{u} = \left( f\left(X_t^{u}\right) + g\left(X_t^{u} \right) u\left(X_t^{u} \right) \right) dt + \sqrt{2} dW_t, \hspace{1em} X^u_0 = x. \nonumber
    \end{align}
    By It\^o's formula and the equation (\ref{eq:lin_HJB}), 
    \begin{align}
        & \quad d \left( e^{-\gamma t} V_u\left(X^u_t\right)\right) \nonumber \\
        & =  e^{-\gamma t} \left( -\gamma V_u + \mathcal{L}_{u} V_u \right) \left(X_t^u \right) dt + \sqrt{2} e^{-\gamma t} \nabla V_u \left(X_t^u \right) dW_t \nonumber \\
        & =  e^{-\gamma t} \left( -\ell - \Vert u \Vert_R^2 \right)\left(X_t^u \right)dt + \sqrt{2} e^{-\gamma t} \nabla V_u \left(X_t^u \right) dW_t. \nonumber
    \end{align}
    Integrate both sides from $0$ to $T$, we get 
    \begin{align}
        & e^{-\gamma T} V_u\left(X_t^u \right) - V_u\left(X^u_0 \right) \nonumber \\
        = & \int_{0 }^{T} e^{-\gamma t}  \left( -\ell - \Vert u \Vert_R^2 \right)\left(X_t^u \right)dt + \sqrt{2} \int_{0 }^{T} e^{-\gamma t} \nabla V_u \left(X_t^u \right) dW_t. \nonumber
    \end{align}
    Taking expectation with limit $T \to \infty$, we obtain
    \begin{align}
        \mathbb{E} \left[\int_{0 }^{\infty} e^{-\gamma t}  \left( \ell + \Vert u \Vert_R^2 \right)\left(X_t^u \right)dt \right] = V_u(x) - \mathbb{E} \left[ \lim_{T \to \infty} e^{-\gamma T} V_u\left(X_t^u \right) \right]. 
    \end{align}
    Since $V_u \in \mathcal{B}^{s+1}(\mathbb{R}^d)$, $V_u$ is bounded on $\mathbb{R}^d$. Hence
    \begin{align}
        \mathbb{E} \left[ \lim_{T \to \infty} e^{-\gamma T} V_u\left(X_t^u \right) \right] = 0. \nonumber
    \end{align}
    Thus 
    \begin{align}
        \mathbb{E} \left[\int_{0 }^{\infty} e^{-\gamma t}  \left( \ell + \Vert u \Vert_R^2 \right)\left(X_t^u \right)dt \right] = V_u(x) < \infty. \nonumber
    \end{align}
    Therefore, $u$ is admissible. 
\end{proof}

Now we construct sequences of controls and value functions in spectral Barron spaces through the policy iteration algorithm. The following proposition also shows that the resulting sequence of value functions is pointwise monotone.

\begin{prop}\label{prop:ply_iter_spec_barron}
    Suppose that Assumption \ref{asp:coef} holds for $s \geq 1$ and the initial control $u^{(0)} \in \mathcal{B}^{s}(\mathbb{R}^d)$. Then the policy iteration (Alg.~\ref{alg:HJB_2nd_ord}) gives a sequence of admissible controls $\{u^{(i)}\}$ and a sequence of value functions $\{V^{(i)}\}$ such that for any $i \in \mathbb{N}$, $u^{(i)} \in \mathcal{B}^{s}(\mathbb{R}^d)$ and $V^{(i)} \in \mathcal{B}^{s+1}(\mathbb{R}^d)$. Moreover, $\{V^{(i)}\}$ is pointwise decreasing. %and $V^{(i)}(x)$ converges pointwise to some $\bar{V}(x)$ for every $x \in \mathbb{R}^d$. 
\end{prop}

\begin{proof}
    1. We show by induction that the policy iteration (Alg.~\ref{alg:HJB_2nd_ord}) generates a sequence of admissible controls $\{u^{(i)}\}$ in $\mathcal{B}^{s}(\mathbb{R}^d)$ and a sequence of value functions $\{V^{(i)}\}$ in $\mathcal{B}^{s+1}(\mathbb{R}^d)$. Suppose $u^{(i)} \in \mathcal{B}^{s}(\mathbb{R}^d)$. By Lemma \ref{lmm:barron_adm}, $u^{(i)}$ is admissible and there exists a unique solution $V^{(i)} \in \mathcal{B}^{s+1}(\mathbb{R}^d)$ to the equation (\ref{alg:1_solve_ghjb_2nd}). By Proposition \ref{prop:barron_alg} (2), $\nabla V^{(i)} \in \mathcal{B}^{s}(\mathbb{R}^d)$. Since $u^{(i+1)}$ is updated according to (\ref{alg:2_update_ctrl_2nd}), by Lemma \ref{lmm:constant_c_R2}, $u^{(i+1)} \in \mathcal{B}^{s}(\mathbb{R}^d)$. Applying Lemma \ref{lmm:barron_adm} again, $u^{(i+1)}$ is admissible. 

    2. We next show that for any $x \in \mathbb{R}^d$, $\{V^{(i)}(x)\}$ is a decreasing sequence. 
    Since $u^{(i+1)}$ is admissible, there exists a unique (strong) solution $X_t^{(i+1)} := X_t^{u^{(i+1)}}$ of the SDE%\jlnotes{try not to use left and right brackets, use $\bigl($ $\Bigl($ etc, so that they are better sized; see an example in the text}
    \begin{align*}
        dX_t^{u^{(i+1)}} = \Bigl( f\bigl(X_t^{u^{(i+1)}}\bigr) + g\bigl(X_t^{u^{(i+1)}} \bigr) u^{(i+1)}\bigl(X_t^{u^{(i+1)}} \bigr) \Bigr) dt + \sqrt{2} dW_t, \qquad X^{u^{(i+1)}}_0 = x. \nonumber
    \end{align*}    
    By It\^o's formula, we obtain
    \begin{align}
        & d \Bigl( e^{-\gamma t} V^{(i)}\bigl(X^{(i+1)}_t \bigr) \Bigr) \nonumber \\
        = & e^{-\gamma t} \Bigl(
        -\gamma V^{(i)} + \mathcal{L}_{u^{(i+1)}} V^{(i)} \Bigr)\bigl(X^{(i+1)}_t \bigr) dt + \sqrt{2} e^{-\gamma t} \nabla V^{(i)}\bigl(X^{(i+1)}_t \bigr) dW_t. \nonumber
    \end{align} 
    For any $T \geq 0$, 
    \begin{align}
        & e^{-\gamma T} V^{(i)}\bigl(X^{(i+1)}_T \bigr) - V^{(i)}\bigl(X^{(i+1)}_0 \bigr) \nonumber \\
        = & \int_{0 }^{T} e^{-\gamma t} \Bigl(
            -\gamma V^{(i)} + \mathcal{L}_{u^{(i+1)}} V^{(i)} \Bigr)\bigl(X^{(i+1)}_t \bigr) dt + \sqrt{2} \int_{0 }^{T}  e^{-\gamma t} \nabla V^{(i)}\bigl(X^{(i+1)}_t \bigr) dW_t. \nonumber
    \end{align}
    Taking expectation with $T \to \infty$ (note $V^{(i)}$ is bounded on $\mathbb{R}^d$), we have 
    \begin{align}
        - V^{(i)}(x) = \mathbb{E} \left[ \int_{0 }^{\infty} e^{-\gamma t} \Bigl(
            -\gamma V^{(i)} + \mathcal{L}_{u^{(i+1)}} V^{(i)} \Bigr)\bigl(X^{(i+1)}_t \bigr) dt \right]. \nonumber
    \end{align}
    Similarly, we have 
    \begin{align}
        - V^{(i+1)}(x) = \mathbb{E} \left[ \int_{0 }^{\infty} e^{-\gamma t} \Bigl(
            -\gamma V^{(i+1)} + \mathcal{L}_{u^{(i+1)}} V^{(i+1)} \Bigr)\bigl(X^{(i+1)}_t \bigr) dt \right]. \nonumber
    \end{align}
    Hence 
    \begin{align}
        & V^{(i+1)}(x) - V^{(i)}(x) \nonumber \\
        = & \mathbb{E} \left[ \int_{0 }^{\infty} e^{-\gamma t} \Bigl( -\gamma V^{(i)} + \mathcal{L}_{u^{(i+1)}} V^{(i)} + \gamma V^{(i+1)} - \mathcal{L}_{u^{(i+1)}} V^{(i+1)} \Bigr)\bigl(X^{(i+1)}_t \bigr) dt \right]. \nonumber
    \end{align}
    We know that 
    \begin{align}
        \nabla V^{(i)}\cdot (f + gu^{(i+1)}) = \gamma V^{(i)} + \nabla V^{(i)} \cdot  g(u^{(i+1)} - u^{(i)}) - (\ell + \Vert u^{(i)} \Vert_R^2) - \Delta V^{(i)}, \nonumber
    \end{align}
    and 
    \begin{align}
        \nabla V^{(i+1)}\cdot (f + gu^{(i+1)}) = \gamma V^{(i+1)}- (\ell + \Vert u^{(i+1)} \Vert_R^2) - \Delta V^{(i+1)}. \nonumber
    \end{align}
    So 
    \begin{align}
        & V^{(i+1)}(x) - V^{(i)}(x) = \nonumber \\
        & = \mathbb{E} \left[ \int_{0 }^{\infty} e^{-\gamma t} \Bigl( \Vert u^{(i+1)} \Vert_R^2 - \Vert u^{(i)} \Vert_R^2 + \nabla V^{(i)} \cdot g(u^{(i+1)} - u^{(i)}) \Bigr) \bigl(X^{(i+1)}_t \bigr) dt \right]. \nonumber
    \end{align}
    Since $u^{(i+1)} = - \frac{1}{2} R^{-1} g^T \nabla V^{(i)}$, 
    \begin{align}
        & V^{(i+1)}(x) - V^{(i)}(x) =\nonumber \\
        & \quad =  \mathbb{E} \left[ \int_{0 }^{\infty} e^{-\gamma t} \Bigl( \Vert u^{(i+1)} \Vert_R^2 - \Vert u^{(i)} \Vert_R^2 + \nabla V^{(i)} \cdot g(u^{(i+1)} - u^{(i)}) \Bigr) \bigl(X^{(i+1)}_t \bigr) dt \right]\nonumber \\
        & \quad =  \mathbb{E} \left[ \int_{0 }^{\infty} e^{-\gamma t} \Bigl( \Vert u^{(i+1)} \Vert_R^2 - \Vert u^{(i)} \Vert_R^2 - 2 \Vert u^{(i+1)} \Vert_R^2 + 2 \langle u^{(i+1)}, u^{(i)}\rangle_R \Bigr) \bigl(X^{(i+1)}_t \bigr) dt \right] \nonumber \\
        & \quad =  \mathbb{E} \left[ \int_{0 }^{\infty} e^{-\gamma t} \Bigl( - \Vert  u^{(i+1)} - u^{(i)} \Vert_R^2 \Bigr) \bigl(X^{(i+1)}_t \bigr) dt \right]\leq 0. \nonumber \qedhere
    \end{align}
    %Therefore, for every $x \in \mathbb{R}^d$, $\{V^{(i)}(x)\}$ is a decreasing sequence bounded below by $V(x)$. Thus $V^{(i)}$ converges pointwise to some $\bar{V}$ on $\mathbb{R}^d$. 
\end{proof}

Proposition \ref{prop:ply_iter_spec_barron} shows that the sequences of controls and value functions lie in the spectral Barron spaces as long as the initial control and the coefficients are spectral Barron functions, and that the sequence of value functions is pointwise decreasing. 
Proposition \ref{prop:ply_iter_spec_barron}, however, does not guarantee that the sequence $\{V^{(i)}\}$ converges. In the following Propositions \ref{prop:bdd_iterates} and \ref{prop:uni_conv_opt}, we prove that when the discount factor is large enough, $\{V^{(i)}\}$ is bounded in spectral Barron norm and thus uniformly bounded on $\mathbb{R}^d$.  As a result, $\{V^{(i)}\}$ decreases pointwise to some limit function $\bar{V}$. In fact, we can show that $\{V^{(i)}\}$ converges locally uniformly to $\bar{V}$ on $\mathbb{R}^d$.  

We first show in Proposition \ref{prop:bdd_iterates} that when the discount factor $\gamma$ is sufficiently large,  if the spectral Barron norm of the initial control is sufficiently small, then the sequences of controls and value functions are both uniformly bounded in spectral Barron spaces. 

\begin{prop}\label{prop:bdd_iterates}
    Suppose that Assumption \ref{asp:coef} holds for $s \geq 1$. Let $C_{R,1}$ and $C_{R,2}$ be the positive constants given in Lemma \ref{lmm:constant_c_R1} and Lemma \ref{lmm:constant_c_R2} respectively. Assume that 
    \begin{align}\label{asp:large_gamma}
        2(\sqrt{1+ \gamma} -1) \geq \Vert f \Vert_{\mathcal{B}^{s}(\mathbb{R}^d)} + 2 \Vert g \Vert_{\mathcal{B}^{s}(\mathbb{R}^d)} \Vert \ell \Vert_{\mathcal{B}^{s}(\mathbb{R}^d)}^{1/2} \left(C_{R,1} C_{R,2}^2 + C_{R,2} \right)^{1/2}. 
    \end{align}
    Define $h(x) = \dfrac{a + b x^2 }{c - dx}$ where 
    \begin{equation*}
        \begin{aligned}[c]
            a & := C_{R,2} \Vert g \Vert_{\mathcal{B}^{s}(\mathbb{R}^d)} \Vert \ell \Vert_{\mathcal{B}^{s}(\mathbb{R}^d)}, \\
            b & := C_{R,1} C_{R,2} \Vert g \Vert_{\mathcal{B}^{s}(\mathbb{R}^d)},
        \end{aligned}
        \hspace{2em}
        \begin{aligned}[c]
            c & := 2(\sqrt{1+ \gamma} -1) - \Vert f \Vert_{\mathcal{B}^{s}(\mathbb{R}^d)}, \\
            d & := \Vert g \Vert_{\mathcal{B}^{s}(\mathbb{R}^d)}.
        \end{aligned}
    \end{equation*}
    Then there exists $x_0 \in (0, c/d)$ such that $x_0 = h(x_0)$. Further, assume that the initial control $u^{(0)} \in \mathcal{B}^{s}(\mathbb{R}^d)$ and satisfies $\Vert u^{(0)} \Vert_{\mathcal{B}^{s}(\mathbb{R}^d)} < x_0$. Then Algorithm \ref{alg:HJB_2nd_ord} gives a sequence of admissible controls $\{u^{(i)}\}$ in $\mathcal{B}^{s}(\mathbb{R}^d)$ and a sequence of value functions $\{V^{(i)}\}$ in $\mathcal{B}^{s+1}(\mathbb{R}^d)$. Moreover, $\{u^{(i)}\}$ and $\{V^{(i)}\}$ are uniformly bounded in $\mathcal{B}^{s}(\mathbb{R}^d)$ and $\mathcal{B}^{s+1}(\mathbb{R}^d)$. Specifically, 
    \begin{align}
        \sup_{i \in \mathbb{N}} \Vert u^{(i)} \Vert_{\mathcal{B}^{s}(\mathbb{R}^d)} \leq x_0, \text{ and } \sup_{i \in \mathbb{N}} \Vert V^{(i)} \Vert_{\mathcal{B}^{s+1}(\mathbb{R}^d)} \leq a_0 h(x_0). \nonumber
    \end{align}
    where $a_0 = \left( C_{R,2} \Vert g \Vert_{\mathcal{B}^{s}(\mathbb{R}^d)} \right)^{-1}$. 
\end{prop}

\begin{proof}
    1. Given the assumption that $\Vert g \Vert_{\mathcal{B}^{s}(\mathbb{R}^d)} > 0$ and the discount factor $\gamma$ satisfies (\ref{asp:large_gamma}), we have $b, c, d > 0$ and $a \geq 0$.  Moreover, (\ref{asp:large_gamma}) entails that
    \begin{align}
        \left( 2(\sqrt{1+ \gamma} -1) - \Vert f \Vert_{\mathcal{B}^{s}(\mathbb{R}^d)}\right)^2 - 4 C_{R,2} \Vert g \Vert_{\mathcal{B}^{s}(\mathbb{R}^d)}^2 \Vert \ell \Vert_{\mathcal{B}^{s}(\mathbb{R}^d)} \left( C_{R,1} C_{R,2} + 1\right) \geq 0, \nonumber
    \end{align}
    that is, $c^2 - 4a(b+d) \geq 0$, which implies that there exists at least one real solution $x_0$ to the equation $x = h(x)$, i.e., $(b+d)x^2 - cx +a = 0$. In fact, $x_0$ is a fixed point of $h(x)$. Since $h$ is positive for all $x < c/d$ and negative for $x > c/d$, we see that $0 \leq x_0 < c/d$. Notice that $h(x)$ is strictly increasing on $[0, x_0)$ because
    \begin{align}
        h'(x) = \dfrac{2bc x - bd x^2 + ad }{(c-dx)^2} \nonumber
    \end{align}
    has two roots $x_1 < 0$ and $x_2 > 0$. There is a local minimum at $x_1 < 0$, and $h(x)$ is strictly increasing on $(x_1, c/d)$, while $h(x) \to \infty$ as $x \to (c/d)^{-}$. In this way, we obtain 
    \begin{align}\label{eq:fix_pt_h}
        0 \leq x < x_0 \hspace{0.5em} \Longrightarrow \hspace{0.5em} 0 \leq h(0) \leq h(x) < h(x_0) = x_0. 
    \end{align}

    2. Fix $s \geq 1$. We start by choosing an inital control $u^{(0)}$ with 
    \begin{align}
        0 \leq \Vert u^{(0)} \Vert_{\mathcal{B}^{s}(\mathbb{R}^d)} < x_0 < c/d = \left( 2(\sqrt{1 + \gamma}-1) - \Vert f \Vert_{\mathcal{B}^{s}(\mathbb{R}^d)}\right) / \Vert g \Vert_{\mathcal{B}^{s}(\mathbb{R}^d)}, \nonumber
    \end{align}
    which allows $u^{(0)}$ to satisfy
    \begin{align}\label{ineq:u_bdd}
        \dfrac{1}{2(\sqrt{1+\gamma} - 1)} \Vert f + g u^{(0)} \Vert_{\mathcal{B}^{s}(\mathbb{R}^d)} < 1. 
    \end{align}
    Since the operator $(\gamma - \Delta)^{-1} (f + g u^{(0)}) \cdot \nabla: \mathcal{B}^{s}(\mathbb{R}^d) \to \mathcal{B}^{s}(\mathbb{R}^d)$ is linear and bounded, and by Proposition \ref{prop:bdd_gamma_delta}, it satisfies
    \begin{align}
        \Vert (\gamma - \Delta)^{-1} (f + gu^{(0)}) \cdot \nabla \Vert_{\mathcal{B}^{s}(\mathbb{R}^d)\to \mathcal{B}^{s}(\mathbb{R}^d)} \leq \dfrac{1}{2(\sqrt{1+\gamma} - 1)} \Vert f + gu^{(0)} \Vert_{\mathcal{B}^{s}(\mathbb{R}^d)}, \nonumber
    \end{align}
    it follows from (\ref{ineq:u_bdd}) that the inverse $\left( I - \left(\gamma - \Delta \right)^{-1} (f + gu^{(0)}) \cdot \nabla \right)^{-1}$ exists with
    \begin{align}
        \Vert \Bigl( I - \left(\gamma - \Delta \right)^{-1} (f + gu^{(0)}) \cdot \nabla \Bigr)^{-1} \Vert_{\mathcal{B}^{s}(\mathbb{R}^d)\to \mathcal{B}^{s}(\mathbb{R}^d)} \leq \dfrac{1 }{1 - C_\gamma \Vert (f + gu^{(0)}) \Vert_{\mathcal{B}^{s}(\mathbb{R}^d)}},  \nonumber
    \end{align}
    where $C_\gamma = \dfrac{1}{2(\sqrt{1+\gamma} - 1)}$. 
    Define 
    \begin{align}
        \mathcal{T}_{u^{(0)}} := \gamma - \Delta - (f + g u^{(0)}) \cdot \nabla. \nonumber
    \end{align}
    Since 
    \begin{align}
        \left( \mathcal{T}_{u^{(0)}} \right)^{-1} = \left( I - \left(\gamma - \Delta \right)^{-1} (f + g u^{(0)}) \cdot \nabla \right)^{-1} (\gamma - \Delta)^{-1}, \nonumber
    \end{align}
    we obtain
    \begin{align}\label{eq:bdd_T}
        \Vert \left( \mathcal{T}_{u^{(0)}} \right)^{-1} \Vert_{\mathcal{B}^{s}(\mathbb{R}^d) \to \mathcal{B}^{s+1}(\mathbb{R}^d)} & \leq \dfrac{C_\gamma }{1 - C_{\gamma} \Vert (f + g u^{(0)}) \Vert_{\mathcal{B}^{s}(\mathbb{R}^d)}} \nonumber \\
        & \leq \dfrac{1 }{2(\sqrt{1+ \gamma} -1) - \Vert (f + g u^{(0)}) \Vert_{\mathcal{B}^{s}(\mathbb{R}^d)}}. 
    \end{align}
    The value function $V^{(0)}$ satisfies
    \begin{align}
        \mathcal{T}_{u^{(0)}} V^{(0)} = \ell + \Vert u^{(0)} \Vert_R^2. \nonumber
    \end{align}
    By Lemma \ref{lmm:constant_c_R1} and (\ref{eq:bdd_T}) we obtain
    \begin{align}
        \Vert V^{(0)} \Vert_{\mathcal{B}^{s+1}(\mathbb{R}^d)} & \leq \Vert \left( \mathcal{T}_{u^{(0)}} \right)^{-1} \bigl( \ell + \Vert u^{(0)} \Vert_R^2 \bigr) \Vert_{\mathcal{B}^{s+1}(\mathbb{R}^d)} \nonumber \\
        & \leq \dfrac{\Vert \ell \Vert_{\mathcal{B}^{s}(\mathbb{R}^d)} + C_{R,1} \Vert u^{(0)} \Vert_{\mathcal{B}^{s}(\mathbb{R}^d)}^2 }{\left(2(\sqrt{1+ \gamma} -1) - \Vert f \Vert_{\mathcal{B}^{s}(\mathbb{R}^d)}\right) - \Vert g \Vert_{\mathcal{B}^{s}(\mathbb{R}^d)} \Vert u^{(0)} \Vert_{\mathcal{B}^{s}(\mathbb{R}^d)}}. \nonumber
    \end{align}
    Then by Lemma \ref{lmm:constant_c_R2} and (\ref{eq:fix_pt_h}), we get 
    \begin{align}
        0 \leq \Vert u^{(1)} \Vert_{\mathcal{B}^{s}(\mathbb{R}^d)} & \leq C_{R,2} \Vert g \Vert_{\mathcal{B}^{s}(\mathbb{R}^d)} \Vert V^{(0)} \Vert_{\mathcal{B}^{s+1}(\mathbb{R}^d)} \nonumber \\
        & \leq \dfrac{ C_{R,2} \Vert g \Vert_{\mathcal{B}^{s}(\mathbb{R}^d)} \Bigl( \Vert \ell \Vert_{\mathcal{B}^{s}(\mathbb{R}^d)} + C_{R,1} \Vert u^{(0)} \Vert_{\mathcal{B}^{s}(\mathbb{R}^d)}^2 \Bigr) }{\left(2(\sqrt{1+ \gamma} -1) - \Vert f \Vert_{\mathcal{B}^{s}(\mathbb{R}^d)}\right) - \Vert g \Vert_{\mathcal{B}^{s}(\mathbb{R}^d)} \Vert u^{(0)} \Vert_{\mathcal{B}^{s}(\mathbb{R}^d)}} \nonumber \\
        & = h\bigl( \Vert u^{(0)} \Vert_{\mathcal{B}^{s}(\mathbb{R}^d)} \bigr) < x_0 < c/d. \nonumber
    \end{align}
    Then we iterate again and get 
    \begin{align}
        \Vert V^{(1)} \Vert_{\mathcal{B}^{s+1}(\mathbb{R}^d)} & \leq \Vert \left( \mathcal{T}_{u^{(1)}} \right)^{-1} \bigl( \ell + \Vert u^{(1)} \Vert_R^2 \bigr) \Vert_{\mathcal{B}^{s+1}(\mathbb{R}^d)} \nonumber \\
        & \leq \dfrac{\Vert \ell \Vert_{\mathcal{B}^{s}(\mathbb{R}^d)} + C_{R,1} \Vert u^{(1)} \Vert_{\mathcal{B}^{s}(\mathbb{R}^d)}^2 }{\left(2(\sqrt{1+ \gamma} -1) - \Vert f \Vert_{\mathcal{B}^{s}(\mathbb{R}^d)}\right) - \Vert g \Vert_{\mathcal{B}^{s}(\mathbb{R}^d)} \Vert u^{(1)} \Vert_{\mathcal{B}^{s}(\mathbb{R}^d)}}, \nonumber
    \end{align}
    and 
    \begin{align}
        0 \leq \Vert u^{(2)} \Vert_{\mathcal{B}^{s}(\mathbb{R}^d)} & \leq C_{R,2} \Vert g \Vert_{\mathcal{B}^{s}(\mathbb{R}^d)} \Vert V^{(1)} \Vert_{\mathcal{B}^{s+1}(\mathbb{R}^d)} \nonumber \\
        & \leq \dfrac{ C_{R,2} \Vert g \Vert_{\mathcal{B}^{s}(\mathbb{R}^d)} \Bigl( \Vert \ell \Vert_{\mathcal{B}^{s}(\mathbb{R}^d)} + C_{R,1} \Vert u^{(1)} \Vert_{\mathcal{B}^{s}(\mathbb{R}^d)}^2 \Bigr) }{\left(2(\sqrt{1+ \gamma} -1) - \Vert f \Vert_{\mathcal{B}^{s}(\mathbb{R}^d)}\right) - \Vert g \Vert_{\mathcal{B}^{s}(\mathbb{R}^d)} \Vert u^{(1)} \Vert_{\mathcal{B}^{s}(\mathbb{R}^d)}} \nonumber \\
        & = h\bigl( \Vert u^{(1)} \Vert_{\mathcal{B}^{s}(\mathbb{R}^d)} \bigr) < x_0 < c/d. \nonumber 
    \end{align}
    Continuing in this way, we obtain that 
    \begin{align}
        \sup_{i \in \mathbb{N}} \Vert u^{(i)} \Vert_{\mathcal{B}^{s}(\mathbb{R}^d)} \leq x_0 < c/d. \nonumber
    \end{align}
    Denote $a_0 := \bigl( C_{R,2} \Vert g \Vert_{\mathcal{B}^{s}(\mathbb{R}^d)} \bigr)^{-1}$, then for each $i \in \mathbb{N}$, 
    \begin{align}
        \Vert V^{(i)} \Vert_{\mathcal{B}^{s+1}(\mathbb{R}^d)} \leq a_0 h \bigl( \Vert u^{(i)} \Vert_{\mathcal{B}^{s}(\mathbb{R}^d)} \bigr) \leq a_0 h(x_0), \nonumber
    \end{align}
    which implies
    \begin{equation*}
        \sup_{i \in \mathbb{N}} \Vert V^{(i)} \Vert_{\mathcal{B}^{s+1}(\mathbb{R}^d)} \leq a_0 h(x_0) < \infty.  \qedhere
    \end{equation*}
\end{proof}

Recall that Proposition \ref{prop:ply_iter_spec_barron} shows that the sequence of spectral Barron value functions $\{V^{(i)}\}$ given by Algorithm \ref{alg:HJB_2nd_ord} decreases pointwise. In proving the next theorem, we show that if the sequence $\{V^{(i)}\}$ is uniformly bounded in spectral Barron norm, then $\{V^{(i)}\}$ decreases to some limit $\bar{V}$. In fact, a subsequence of $\{V^{(i)}\}$ converges locally uniformly to the limit $\bar{V}$ and $\bar{V}$ is a solution to the HJB equation (\ref{eq:HJB_qd}). 

\begin{prop}\label{prop:uni_conv_opt}
    Suppose that Assumption \ref{asp:coef} holds for $s \geq 2$. Under the same assumption as Proposition \ref{prop:bdd_iterates}, there exists a sequence $\{V^{(i)}\}$ in $\mathcal{B}^{s+1}(\mathbb{R}^d)$ and $\bar{V} \in C^2(\mathbb{R}^d)$ such that $\{ V^{(i)}\}$ converges to $\bar{V}$ in $C^2_{\mathrm{loc}}(\mathbb{R}^d)$, and the sequence $u^{(i)} = -\frac{1}{2} R^{-1} g^T \nabla V^{(i-1)}$ converges to $\bar{u} = - \frac{1}{2} R^{-1} g^{T} \nabla \bar{V}$ in $C^1_{\mathrm{loc}}(\mathbb{R}^d)$. 
    As a result, $\bar{V}$ is a classical solution to the HJB equation (\ref{eq:HJB_qd}) on $\mathbb{R}^d$. 
\end{prop}

\begin{proof}
    By Proposition \ref{prop:bdd_iterates}, Algorithm \ref{alg:HJB_2nd_ord} gives a sequence of admissible controls $\{u^{(i)}\}$ in $\mathcal{B}^{s}(\mathbb{R}^d)$ and a sequence of value functions $\{V^{(i)}\}$ in $\mathcal{B}^{s+1}(\mathbb{R}^d)$. Moreover, we have $\sup_{i \in \mathbb{N}}\Vert V^{(i)} \Vert_{\mathcal{B}^{s+1}(\mathbb{R}^d)} < \infty$. Thus $\sup_{i \in \mathbb{N}} \Vert \widehat{V^{(i)}} \Vert_{L^1(\mathbb{R}^d)} < \infty$. Then there exists $M_0 > 0$ such that for all $i \in \mathbb{N}$, 
    \begin{align}
        |V^{(i)}(x)| \leq \int_{\mathbb{R}^d} |\widehat{V^{(i)}}(\xi)| \dd \xi < M_0, \hspace{1em} \forall x \in \mathbb{R}^d. \nonumber
    \end{align}
    Thus $\{V^{(i)}\}$ is uniformly bounded on $\mathbb{R}^d$. In particular, $\{V^{(i)}\}$ is uniformly bounded below. Recall that Proposition \ref{prop:ply_iter_spec_barron} shows that $\{V^{(i)}\}$ is pointwise decreasing, hence $\{V^{(i)}\}$ converges pointwise to some limit $\bar{V}$ on $\mathbb{R}^d$. Further, by Proposition \ref{prop:barron_alg} (2), we also have
    \begin{align}
        \sup_{i \in \mathbb{N}}\Vert \nabla^m V^{(i)} \Vert_{\mathcal{B}^{s+1-m}(\mathbb{R}^d)} < \infty, \text{ for } m = 1, 2, 3. \nonumber
    \end{align}
    Recall that $s \geq 2$, hence there also exist $M_1, M_2, M_3 > 0$ such that 
    \begin{align}
        \sup_{i \in \mathbb{N}} \sup_{x \in \mathbb{R}^d} |\nabla V^{(i)}(x)| < M_1, \hspace{.5em} \sup_{i \in \mathbb{N}} \sup_{x \in \mathbb{R}^d} |\nabla^2 V^{(i)}(x)| < M_2, \hspace{.5em} \sup_{i \in \mathbb{N}} \sup_{x \in \mathbb{R}^d} |\nabla^3 V^{(i)}(x)| < M_3. \nonumber
    \end{align}
    For any $i \in \mathbb{N}$ and any $x, y \in \mathbb{R}^d$, we have 
    \begin{align}
        |V^{(i)}(x) - V^{(i)}(y)| \leq M_1 |x - y|. \nonumber
    \end{align}
    This shows that $\{V^{(i)}\}$ is equicontinuous on $\mathbb{R}^d$. By the Arzela-Ascoli theorem, on any compact domain $K \subset \mathbb{R}^d$, there exists a subsequence of $\{V^{(i)}\}$ that converges uniformly to some continuous function $\bar{V}_{K}$ on $K$. Since $V^{(i)}$ converges pointwise to $\bar{V}$ on $\mathbb{R}^d$, we have $\bar{V} = \bar{V}_{K}$ on $K$. Since $\mathbb{R}^d = \bigcup_{j \in \mathbb{N}} \overline{B_j(0)}$, by applying the diagonalization argument, one can extract a subsequence of $\{V^{(i)}\}$ that converges uniformly to $\bar{V}$ in every $\overline{B_j(0)}$ and thus converges uniformly to $\bar{V}$ on any compact domain $ K \subset \mathbb{R}^d$. With a little abuse of notation, we denote the subsequence as $\{V^{(i)}\}$ itself. Therefore, $\bar{V}$ is continuous and $\{V^{(i)}\}$ converges locally uniformly to $\bar{V}$ on $\mathbb{R}^d$. 
    
    Similarly, for any $i \in \mathbb{N}$ and $x, y \in \mathbb{R}^d$,
    \begin{align}
        |\nabla V^{(i)}(x) - \nabla V^{(i)}(y)| \leq M_2 |x - y|. \nonumber
    \end{align} 
    Thus $\{\nabla V^{(i)}\}$ is also uniformly bounded and equicontinuous on $\mathbb{R}^d$. By the Arzela-Ascoli theorem, for any compact domain $K \subset \mathbb{R}^d$, there exists a subsequence of $\{\nabla V^{(i)}\}$ that converges uniformly on $K$. Since $\{V^{(i)}\}$ converges uniformly to $\bar{V}$ on $K$, it follows that $\bar{V}$ is differentiable and the subsequence of $\{\nabla V^{(i)}\}$ converges uniformly to $\nabla \bar{V}$ on $K$. Moreover, $\nabla \bar{V}$ is continuous on $K$ as $\nabla V^{(i)}$'s are continuous. Since $K$ is arbitrary, $\bar{V}$ is continuously differentiable on $\mathbb{R}^d$. By applying the diagonalization argument as before, we assert that there exists a subsequence of $\nabla V^{(i)}$, still denoted as $\nabla V^{(i)}$ itself, that converges locally uniformly to $\nabla \bar{V}$ on $\mathbb{R}^d$. 

    By repeating the argument for $\{\nabla^2 V^{(i)}\}$, one obtains that $\bar{V}$ is twice continuously differentiable and, by extracting a subsequence if necessary, $\{\nabla^2 V^{(i)}\}$ that converges locally uniformly to $\nabla^2 \bar{V}$ on $\mathbb{R}^d$. Define 
    \begin{align}
        \bar{u}(x) := -\frac{1}{2} R^{-1} g(x)^T \nabla \bar{V}(x). \nonumber
    \end{align}
    Then the sequence of controls $\{u^{(i)}\}$ converges locally uniformly to $\bar{u}$ on $\mathbb{R}^d$. Notice that $g \in \mathcal{B}^{s}(\mathbb{R}^d)$ for $s \geq 2$ implies $g$ is continuously differentiable on $\mathbb{R}^d$, hence $\bar{u}$ is continuously differentiable and $\{ \nabla u^{(i)} \}$ converges locally uniformly to $\nabla \bar{u}$ on $\mathbb{R}^d$ as well. Taking the limit $i \to \infty$ in the equation
    \begin{align}
        -\gamma V^{(i)}(x) + \mathcal{L}_{u^{(i)}} V^{(i)}(x) + \ell(x) + \Vert u^{(i)}(x) \Vert_R^2 = 0 \nonumber
    \end{align}
    gives 
    \begin{align}
        -\gamma \bar{V}(x) + \mathcal{L}_{\bar{u}} \bar{V}(x) + \ell(x) + \Vert \bar{u}(x) \Vert_R^2 = 0,  \hspace{1em} x\in \mathbb{R}^d. \nonumber
    \end{align}
    Therefore, $\bar{V}$ satisfies (\ref{eq:HJB_equiv}) and thus is a classical solution of the HJB equation (\ref{eq:HJB_qd}) in $\mathbb{R}^d$. 
\end{proof}

\begin{remark}
    In proving Proposition \ref{prop:uni_conv_opt}, the reason we require $\mathcal{B}^s$ regularity for $s \geq 2$ is to obtain the uniform bound of $\{\nabla^3 V^{(i)}\}$ in the spectral Barron space, which then guarantees the equicontinuity of $\{\nabla^2 V^{(i)}\}$. 
\end{remark}

\begin{proof}[Proof of Theorem \ref{thm:main_uni_conv}]
    It is apparent from (\ref{asp:large_gamma}) that there exists a constant $C(f,g, \ell,R)$ depending on the coefficients $f,g, \ell,R$ such that $\gamma \geq C(f,g, \ell,R)$ implies $\gamma$ satisfies (\ref{asp:large_gamma}). The statement of Theorem \ref{thm:main_uni_conv} is a direct consequence of Proposition \ref{prop:uni_conv_opt}. 
\end{proof}

\begin{remark}
    In the proofs of Propositions \ref{prop:bdd_iterates}, \ref{prop:uni_conv_opt}, and Theorem \ref{thm:main_uni_conv}, the assumption that $\gamma$ is sufficiently large is used to ensure that the policy iteration sequence remains bounded in the spectral Barron space, which is essential for our analysis. Whether comparable approximation results can be established when $\gamma$ is small remains unclear. We expect that this regime presents additional analytical challenges and may require new techniques beyond those employed here, and will be left as an open research direction. 
\end{remark}

\begin{proof}[Proof of Theorem \ref{thm:main_nn_appx}]
    By Proposition \ref{prop:uni_conv_opt}, the sequence $\{\nabla^k u^{(i)}\}$ converges uniformly to $\nabla^k \bar{u}$ on $K$ for $k = 0,1$. Let $n \in \mathbb{N}$. Then there exists $N \in \mathbb{N}$ such that 
    \begin{align}
        \sup_{x \in K} |\nabla^k u^{(N)}(x) - \nabla^k \bar{u}(x)| < \dfrac{1}{2 \sqrt{n}} \text{ for } k = 0,1. \nonumber
    \end{align}
    Note that $|K|$ is finite. This implies 
    \begin{align}
        \Vert u^{(N)} - \bar{u} \Vert_{H^1(K)} \leq \sqrt{|K|} n^{-1/2}. \nonumber
    \end{align}
    Since $u^{(N)} \in \mathcal{B}^{1}(\mathbb{R}^d)$, by Proposition \ref{prop:nn_apprx_spec_barron}, there exists a cosine-activated two-layer neural network $u_n$ with $n$ hidden neurons 
    \begin{align}
        u_n(x) = \dfrac{1}{n} \sum_{j=1}^{n} a_j \cos(w_j \cdot x + b_j), \nonumber
    \end{align}
    where $a_j, b_j \in \mathbb{R}$ and $w_j \in \mathbb{R}^d$ for each $j$, such that 
    \begin{align}
        \Vert u_n - u^{(N)} \Vert_{H^1(K)} \leq \sqrt{|K|} \Vert u^{(N)} \Vert_{\mathcal{B}^{1}(\mathbb{R}^d)} n^{-1/2}. \nonumber
    \end{align}
    By Proposition \ref{prop:bdd_iterates}, $\sup_{i \in \mathbb{N}} \Vert u^{(i)} \Vert_{\mathcal{B}^{1}(\mathbb{R}^d)} < \infty$, so we obtain
    \begin{align}
        \Vert u_n - u^{(N)} \Vert_{H^1(K)} \leq \sqrt{|K|} \Bigl( \sup_{i \in \mathbb{N}} \Vert u^{(i)} \Vert_{\mathcal{B}^{1}(\mathbb{R}^d)} \Bigr) n^{-1/2}.
    \end{align}
    Hence, 
    \begin{align}
        \Vert u_n - \bar{u} \Vert_{H^1(K)} \leq \sqrt{|K|} \Bigl(1 + \sup_{i \in \mathbb{N}} \Vert u^{(i)} \Vert_{\mathcal{B}^{1}(\mathbb{R}^d)}  \Bigr) n^{-1/2}. \nonumber
    \end{align}
    
    Similarly, by Proposition \ref{prop:uni_conv_opt}, $\{\nabla^m V^{(i)}\}$ converges uniformly to $\nabla^m \bar{V}$ on $\bar{K}$ for $m = 0,1,2$. Then there exists a sufficiently large $N$ such that 
    \begin{align}
        \sup_{x \in K} |\nabla^m V^{(N)}(x) - \nabla^m \bar{V}(x)| < \dfrac{1}{3 \sqrt{n}} \text{ for } m = 0,1,2. \nonumber
    \end{align}
    This implies 
    \begin{align}
        \Vert V^{(N)} - \bar{V} \Vert_{H^2(K)} \leq \sqrt{|K|} n^{-1/2}. \nonumber
    \end{align}
    Since $V^{(N)} \in \mathcal{B}^{2}(\mathbb{R}^d)$, by Proposition \ref{prop:nn_apprx_spec_barron}, there exists a cosine-activated two-layer neural network $V_n$ with $n$ neurons 
    \begin{align}
        V_n(x) = \dfrac{1}{n} \sum_{j=1}^{n} c_j \cos(v_j \cdot x + d_j), \nonumber
    \end{align}
    where $c_j, d_j \in \mathbb{R}$ and $v_j \in \mathbb{R}^d$ such that 
    \begin{align}
        \Vert V_n - V^{(N)} \Vert_{H^2(K)} \leq \sqrt{|K|} \Vert V^{(N)} \Vert_{\mathcal{B}^2(\mathbb{R}^d)} n^{-1/2}. \nonumber
    \end{align}
    By Proposition \ref{prop:bdd_iterates}, $\sup_{i \in \mathbb{N}} \Vert V^{(i)} \Vert_{\mathcal{B}^{2}(\mathbb{R}^d)} < \infty$. Therefore, 
    \begin{equation*}
        \Vert V_n - \bar{V} \Vert_{H^2(K)} \leq \sqrt{|K|} \Bigl(1 + \sup_{i \in \mathbb{N}} \Vert V^{(i)} \Vert_{\mathcal{B}^{2}(\mathbb{R}^d)}  \Bigr) n^{-1/2}. \qedhere
    \end{equation*}
\end{proof}

\appendix
\section{Regularity of Linear Second-order Elliptic PDEs in Spectral Barron Spaces} \label{app:reg_elliptic}

This section aims to prove the regularity theorem of linear second-order elliptic PDEs in spectral Barron spaces. The main idea is to apply the Fredholm alternative. Similar techniques are used in \cite{chen2023regularity}. 
The following Kolmogorov-Riesz theorem will be used to prove the compactness of the operator $(\gamma I - \Delta)^{-1} \mu \cdot \nabla$ in the spectral Barron space. The proof is standard and can be found, for example, in \cite{hanche2010kolmogorov}. 

\begin{theorem}[Kolmogorov-Riesz]
    For $1 \leq p < \infty$, a subset $\mathcal{F} \subset L^p(\mathbb{R}^d)$ is totally bounded if and only if the following three conditions hold:
    \begin{itemize}
        \item[(i)] $\mathcal{F}$ is bounded.
        \item[(ii)] For any $\epsilon > 0$, there exists $R > 0$ for which 
        \begin{align}
            \int_{|x| > R} |f(x)|^p \dd x < \epsilon^p, \hspace{1em} \forall f \in \mathcal{F}. \nonumber
        \end{align}
        \item[(iii)] For any $\epsilon > 0$, there exists $\delta > 0$ for which 
        \begin{align}
            \int_{\mathbb{R}^d} |f(x+y) - f(x)|^p \dd x < \epsilon^p, \hspace{1em} \forall f \in \mathcal{F}, \forall |y| < \delta. \nonumber
        \end{align}
    \end{itemize}
\end{theorem}

We then apply the Kolmogorov-Riesz theorem to show that the operator $(\gamma I - \Delta)^{-1} \mu \cdot \nabla$ is compact.  

\begin{lemma}\label{lmm:compact}
    For any $s \geq 1$, assume that $\gamma > 0$ and $\mu : \mathbb{R}^d \to \mathbb{R}^d \in \mathcal{B}^{s-1}(\mathbb{R}^d)$. Then the operator $\mathcal{T} = (\gamma I - \Delta)^{-1} \mu \cdot \nabla: \mathcal{B}^{s}(\mathbb{R}^d) \to \mathcal{B}^{s}(\mathbb{R}^d)$ is compact. 
\end{lemma}

\begin{proof}
    To show that $\mathcal{T}: \mathcal{B}^{s}(\mathbb{R}^d) \to \mathcal{B}^s(\mathbb{R}^d)$ is compact, it suffices to show that the image of the closed unit ball under $\mathcal{T}$ is relatively compact in $\mathcal{B}^{s}(\mathbb{R}^d)$. Since $\mathcal{B}^s(\mathbb{R}^d)$ is complete, any subset of $\mathcal{B}^s(\mathbb{R}^d)$ is relatively compact if and only it is totally bounded. Therefore, to show that $\mathcal{T}$ is compact, it suffices to show that the subset 
    \begin{align}
        \mathcal{F} := \left\{ \widehat{\mathcal{T}(v)}(\xi) \left(1 + |\xi| \right)^s \mid \Vert v \Vert_{\mathcal{B}^{s}(\mathbb{R}^d)} \leq 1 \right\} \subset L^1(\mathbb{R}^d) \nonumber
    \end{align}
    is totally bounded. We verify the three conditions of the Kolmogorov-Riesz theorem for $\mathcal{F}$.

    (i) For any $\Vert v \Vert_{\mathcal{B}^{s}(\mathbb{R}^d)} \leq 1$, by Propositions \ref{prop:barron_embed}, \ref{prop:barron_alg} and \ref{prop:bdd_gamma_delta}, 
    \begin{align}
        \Vert \widehat{\mathcal{T}(v)}(\xi) \left(1 + |\xi|\right)^s \Vert_{L^1(\mathbb{R}^d)} = \Vert (\gamma I - \Delta)^{-1} \mu \cdot \nabla v \Vert_{\mathcal{B}^s(\mathbb{R}^d)} \leq C_{\gamma} \Vert \mu \cdot \nabla v\Vert_{\mathcal{B}^{s-2}(\mathbb{R}^d)} \leq C_{\gamma} \Vert \mu \Vert_{\mathcal{B}^{s-1}(\mathbb{R}^d)}, \nonumber
    \end{align}
    where $C_{\gamma} = 1 + \gamma^{-1}$. 

    (ii) For any $\epsilon > 0$, choose $R > 0$ such that 
    \begin{align}
        \dfrac{1 + |\xi| }{\gamma + |\xi|^2} < \epsilon, \hspace{1em} \forall |\xi| > R. \nonumber
    \end{align}
    Then for any $\Vert v \Vert_{\mathcal{B}^{s}(\mathbb{R}^d)} \leq 1$, by Propositions \ref{prop:barron_embed} and \ref{prop:barron_alg}
    \begin{align}
        \int_{|\xi| > R} & |\widehat{\mathcal{T}(v)} (\xi)| \left(1 + |\xi| \right)^{s} d\xi = \int_{|\xi| > R} \dfrac{1 }{\gamma + |\xi|^2 } | \widehat{\mu \cdot \nabla v }(\xi)| \left(1 + |\xi| \right)^s d\xi \nonumber \\
        = &\int_{|\xi| > R} \dfrac{1 + |\xi|}{\gamma + |\xi|^2 } | \widehat{\mu \cdot \nabla v }(\xi)| \left(1 + |\xi| \right)^{s-1} d\xi \leq \epsilon \Vert \mu \cdot \nabla v \Vert_{\mathcal{B}^{s-1}(\mathbb{R}^d)} \leq \epsilon  \Vert \mu \Vert_{\mathcal{B}^{s-1}(\mathbb{R}^d)}. \nonumber
    \end{align}

    (iii) For any $\epsilon > 0$, by (ii) there exists $R > 0$ such that for any $\Vert v \Vert_{\mathcal{B}^{s}(\mathbb{R}^d)} \leq 1$, 
    \begin{align}
        \int_{|\xi| > R} | \widehat{\mathcal{T}(v)}(\xi)| \left(1 + |\xi|\right)^s \dd \xi \leq \epsilon. \nonumber
    \end{align}
    Set 
    \begin{align}
        L_1 := \max_{|\xi| \leq 2R} \dfrac{\left(1 + |\xi| \right)^s}{\gamma + |\xi|^2}, \hspace{1em} L_2 := \int_{|\xi| \leq 2R} \dd \xi = \mathrm{Vol}(B_{2R}(0)). \nonumber
    \end{align}
    Since $\mu \in \mathcal{B}^{s-1}(\mathbb{R}^d)$ for $s \geq 1$, we have $\mu \in \mathcal{B}^{0}(\mathbb{R}^d)$. Denote $\mu = (\mu_1, ..., \mu_d)$, then for each $1 \leq i \leq d$ we have $\widehat{\mu_i} \in L^1(\mathbb{R}^d)$. Hence there exists $\varphi_i \in C_c^{\infty}(\mathbb{R}^d)$ for all $1 \leq i \leq d$ such that 
    \begin{align}
        \sup_{1 \leq i \leq d} \Vert \widehat{\mu_i} - \varphi_i \Vert_{L^1(\mathbb{R}^d)} \leq \dfrac{\epsilon}{L_1}. \nonumber
    \end{align}
    Since $\xi \mapsto \dfrac{\left(1 + |\xi| \right)^s}{\gamma + |\xi|^2}$ is continuous, it is uniformly continuous on any compact subset of $\mathbb{R}^d$.
    Also, each $\varphi_i \in C^{\infty}_c(\mathbb{R}^d)$ and is uniformly continuous on $\mathbb{R}^d$. Thus, there exists $\delta < R$ such that 
    \begin{align}
        \left| \dfrac{\left(1 + |\xi| \right)^s}{\gamma + |\xi|^2} - \dfrac{\left(1 + |\xi'| \right)^s}{\gamma + |\xi'|^2} \right| \leq \epsilon, \hspace{1em} \forall \xi, \xi': |\xi| \leq 3R, |\xi'| \leq 3R, |\xi - \xi'| < \delta, \nonumber
    \end{align}
    and 
    \begin{align}
        \sup_{1 \leq i \leq d} \left| \varphi_i(\xi) - \varphi_i(\xi')\right| \leq \dfrac{\epsilon }{L_1 L_2}, \hspace{1em} \forall \xi, \xi': |\xi - \xi'| < \delta. \nonumber
    \end{align}
    For any $|y| < \delta < R$ and $\Vert v \Vert_{\mathcal{B}^{s}(\mathbb{R}^d)} \leq 1$, 
    \begin{align}\label{eq:cpt_Rd}
        & \int_{\mathbb{R}^d} \left| \widehat{\mathcal{T}(v)}(\xi + y) \left( 1 + |\xi + y| \right)^s - \widehat{\mathcal{T}(v)}(\xi ) \left( 1 + |\xi| \right)^s\right| \dd \xi \nonumber \\
        \leq & \int_{|\xi| > 2R} \left|\widehat{\mathcal{T}(v)}(\xi + y)   \right|\left( 1 + |\xi + y| \right)^s \dd \xi + \int_{|\xi| > 2R} \left|\widehat{\mathcal{T}(v)}(\xi)  \right|\left( 1 + |\xi| \right)^s \dd \xi  \nonumber \\
        & + \int_{|\xi| \leq 2 R} \left| \widehat{\mathcal{T}(v)}(\xi + y) \left( 1 + |\xi + y| \right)^s - \widehat{\mathcal{T}(v)}(\xi ) \left( 1 + |\xi| \right)^s\right| \dd \xi \nonumber \\
        \leq & 2 \epsilon + \int_{|\xi| \leq 2 R} \left| \widehat{\mathcal{T}(v)}(\xi + y) \left( 1 + |\xi + y| \right)^s - \widehat{\mathcal{T}(v)}(\xi ) \left( 1 + |\xi| \right)^s\right| \dd \xi.
    \end{align}
     Note that 
    \begin{align}\label{eq:tau_u}
        & \left| \widehat{\mathcal{T}(v)}(\xi + y) \left( 1 + |\xi + y| \right)^s - \widehat{\mathcal{T}(v)}(\xi ) \left( 1 + |\xi| \right)^s\right| \nonumber \\
        \leq & \left| \widehat{\mu \cdot \nabla v}(\xi + y) \dfrac{(1 +|\xi + y|)^s}{\gamma + |\xi + y|^2} - \widehat{\mu \cdot \nabla v}(\xi) \dfrac{(1 +|\xi |)^s}{\gamma + |\xi|^2} \right| \nonumber \\
        \leq & \left| \widehat{\mu \cdot \nabla v}(\xi + y) \right| \cdot \left| \dfrac{(1 +|\xi + y|)^s}{\gamma + |\xi + y|^2} - \dfrac{(1 +|\xi |)^s}{\gamma + |\xi|^2} \right| + \dfrac{(1 +|\xi |)^s}{\gamma + |\xi|^2}\cdot \left| \widehat{\mu \cdot \nabla v}(\xi + y) - \widehat{\mu \cdot \nabla v}(\xi) \right|. 
    \end{align}
    For the first term of (\ref{eq:tau_u}), we have the following estimate.
    \begin{align}\label{eq:cpt_2R_1}
        & \int_{|\xi| \leq 2 R} \left| \widehat{\mu \cdot \nabla v}(\xi + y) \right| \cdot \left| \dfrac{(1 +|\xi + y|)^s}{\gamma + |\xi + y|^2} - \dfrac{(1 +|\xi |)^s}{\gamma + |\xi|^2} \right| \dd \xi \nonumber \\
        \leq & \epsilon \int_{|\xi| \leq 2 R} \left| \widehat{\mu \cdot \nabla v}(\xi + y) \right| \dd \xi \nonumber \\
        \leq & \epsilon \Vert \mu \cdot \nabla v \Vert_{\mathcal{B}^{0}(\mathbb{R}^d)} \leq \epsilon \Vert \mu \Vert_{\mathcal{B}^{0}(\mathbb{R}^d)}  \Vert \nabla v \Vert_{\mathcal{B}^{0}(\mathbb{R}^d)} \leq \epsilon \Vert \mu \Vert_{\mathcal{B}^{0}(\mathbb{R}^d)} \Vert v \Vert_{\mathcal{B}^{1}(\mathbb{R}^d)} \leq \epsilon \Vert \mu \Vert_{\mathcal{B}^{0}(\mathbb{R}^d)}. 
    \end{align}
    For the second term of (\ref{eq:tau_u}), we estimate as follows.
    \begin{align}\label{eq:cpt_2R_2}
        & \int_{|\xi| \leq 2 R} \dfrac{(1 +|\xi |)^s}{\gamma + |\xi|^2}\cdot \left| \widehat{\mu \cdot \nabla v}(\xi + y) - \widehat{\mu \cdot \nabla v}(\xi) \right| \dd \xi \nonumber \\
        \leq & L_1 \int_{|\xi| \leq 2 R} \sum_{i=1}^{d} \left| \int_{\mathbb{R}^d} \widehat{\partial_i v}(\eta) \widehat{\mu_i}(\xi + y - \eta) \dd \eta - \int_{\mathbb{R}^d} \widehat{\partial_i v}(\eta) \widehat{\mu_i}(\xi - \eta) \dd \eta \right| \dd \xi \nonumber \\
        \leq & L_1 \int_{|\xi| \leq 2 R}  \sum_{i=1}^{d} \bigg\{ \int_{\mathbb{R}^d} |\widehat{\partial_i v}(\eta)| \bigg( |\widehat{\mu_i}(\xi + y - \eta) - \varphi_i(\xi + y - \eta)| + |\widehat{\mu_i}(\xi - \eta) - \varphi_i(\xi - \eta)| \nonumber \\
        & + |\varphi_i(\xi + y - \eta) - \varphi_i(\xi - \eta)| \bigg) \dd \eta \bigg\} \dd \xi \nonumber \\
        \leq & \sum_{i=1}^{d} \bigg\{ \dfrac{\epsilon}{L_2} \int_{|\xi| \leq 2 R} \dd \xi \int_{\mathbb{R}^d} |\widehat{\partial_i v}(\eta)| \dd \eta + 2 L_1 \Vert \widehat{\mu_i} - \varphi_i \Vert_{L^1(\mathbb{R}^d)} \int_{\mathbb{R}^d} |\widehat{\partial_i v}(\eta)| \dd \eta \bigg\} \nonumber \\
        \leq & 3 \epsilon \sum_{i=1}^{d} \int_{\mathbb{R}^d} |\widehat{\partial_i v}(\eta)| \dd \eta \leq 3 \epsilon \Vert \nabla v \Vert_{\mathcal{B}^{0}(\mathbb{R}^d)} \leq 3 \epsilon \Vert v \Vert_{\mathcal{B}^{1}(\mathbb{R}^d)} \leq 3 \epsilon. 
    \end{align} 
    Combining (\ref{eq:cpt_2R_1}) and (\ref{eq:cpt_2R_2}), 
    \begin{align}
        \int_{|\xi| \leq 2 R} \left| \widehat{\mathcal{T}(v)}(\xi + y) \left( 1 + |\xi + y| \right)^s - \widehat{\mathcal{T}(v)}(\xi ) \left( 1 + |\xi| \right)^s\right| \dd \xi \leq  (\Vert \mu \Vert_{\mathcal{B}^{0}(\mathbb{R}^d)}+3) \epsilon, \nonumber
    \end{align}
    and by (\ref{eq:cpt_Rd}), 
    \begin{align}
        \int_{\mathbb{R}^d} \left| \widehat{\mathcal{T}(v)}(\xi + y) \left( 1 + |\xi + y| \right)^s - \widehat{\mathcal{T}(v)}(\xi ) \left( 1 + |\xi| \right)^s\right| \dd \xi \leq  (\Vert \mu \Vert_{\mathcal{B}^{0}(\mathbb{R}^d)}+5) \epsilon \nonumber
    \end{align}
    for any $|y| < \delta$ and $\Vert v \Vert_{\mathcal{B}^{s}(\mathbb{R}^d)} \leq 1$. By the Kolmogorov-Riesz theorem, $\mathcal{F}$ is totally bounded. 
\end{proof}

\begin{lemma}\label{lmm:injective}
    For any $s \geq 1$, suppose that the vector-valued function $\mu : \mathbb{R}^d \to \mathbb{R}^d \in \mathcal{B}^{s-1}(\mathbb{R}^d)$. Then the operator $I + \mathcal{T}: \mathcal{B}^{s}(\mathbb{R}^d) \to \mathcal{B}^{s}(\mathbb{R}^d)$ is injective. 
\end{lemma}

\begin{proof}
    Suppose there exists $v \in \mathcal{B}^{s}(\mathbb{R}^d)$ such that 
    \begin{align}
        v + \mathcal{T}(v) = 0 \text{ in } \mathbb{R}^d, \nonumber
    \end{align}
    which is equivalent to 
    \begin{align}
        \gamma v - \Delta v + \mu \cdot \nabla v = 0 \text{ in } \mathbb{R}^d. \nonumber
    \end{align}
    Suppose there exists $x_0 \in \mathbb{R}^d$ such that $v(x_0) \ne 0$. Then $|v(x_0)| > \epsilon_0 > 0$ for some $\epsilon_0 > 0$. Now since $v \in \mathcal{B}^{s}(\mathbb{R}^d)$ with $s \geq 1$, $\hat{v} \in L^1(\mathbb{R}^d)$. By the Riemann-Lebesgue lemma, $|v(x)| \to 0$ as $|x| \to \infty$. Then there exists $R > 0$ for which $|x_0| \leq R$ and $|v(x)| < \epsilon_0 < |v(x_0)|$ for all $|x| > R$. Choose $r > 0$, Proposition \ref{prop:nn_apprx_spec_barron} implies that $v \in H^1(B_{R+r}(0))$ as $v \in \mathcal{B}^1(\mathbb{R}^d)$. According to the weak maximum principle (see Theorem 8.1 in \cite{gilbarg1977elliptic}), 
    \begin{align}
        \sup_{|x| \leq R+r} |v(x)| = \sup_{|x| = R+r} |v(x)|. \nonumber
    \end{align}
    Therefore, there exists $x_r$ such that $|x_r| = R+r$ and $|v(x_r)| \geq |v(x_0)|$, a contradiction with $|v(x_r)| < \epsilon_0 < |v(x_0)|$. Thus $v \equiv 0$. 
\end{proof}

\begin{coro}\label{cor:bdd_inv}
    For any $s \geq 1$, suppose that the vector-valued function $\mu : \mathbb{R}^d \to \mathbb{R}^d \in \mathcal{B}^{s-1}(\mathbb{R}^d)$. Then the operator $I + \mathcal{T}: \mathcal{B}^{s}(\mathbb{R}^d) \to \mathcal{B}^{s}(\mathbb{R}^d)$ has bounded inverse. 
\end{coro}

\begin{proof}
    By Lemmas \ref{lmm:compact} and \ref{lmm:injective}, the result follows immediately from the Fredholm alternative. 
\end{proof}

\begin{proof}[Proof of Theorem \ref{thm:reg}]
    Let $v^* := (I + \mathcal{T})^{-1} (\gamma I - \Delta)^{-1} v_f$ be the unique solution. Then we have 
    \begin{align}
        \Vert v^* \Vert_{\mathcal{B}^{s+1}(\mathbb{R}^d)} & \leq \Vert (I + \mathcal{T})^{-1} \Vert_{\mathcal{B}^{s+1}(\mathbb{R}^d) \to \mathcal{B}^{s+1}(\mathbb{R}^d)} \Vert (\gamma I - \Delta)^{-1} v_f \Vert_{\mathcal{B}^{s+1}(\mathbb{R}^d)} \nonumber \\
        & \leq \Vert (I + \mathcal{T})^{-1} \Vert_{\mathcal{B}^{s+1}(\mathbb{R}^d) \to \mathcal{B}^{s+1}(\mathbb{R}^d)} \Vert (\gamma I - \Delta)^{-1} v_f \Vert_{\mathcal{B}^{s+2}(\mathbb{R}^d)} \nonumber \\
        & \leq  (1 + \gamma^{-1}) \Vert (I + \mathcal{T})^{-1} \Vert_{\mathcal{B}^{s+1}(\mathbb{R}^d) \to \mathcal{B}^{s+1}(\mathbb{R}^d)} \Vert v_f \Vert_{\mathcal{B}^{s}(\mathbb{R}^d)}, \nonumber
    \end{align}
    where $\Vert (I + \mathcal{T})^{-1} \Vert_{B^{s+1}(\mathbb{R}^d) \to \mathcal{B}^{s+1}(\mathbb{R}^d)}  < \infty$ by Corollary \ref{cor:bdd_inv}. 
\end{proof}

\section{Omitted Proofs in Section \ref{sec:proofs}} \label{app:proofs_conv}

This appendix collects all the omitted proofs in Section~\ref{sec:proofs}. The proof of Proposition \ref{prop:nn_apprx_spec_barron} uses similar techniques as in \cite{chen2023regularity}. 

\begin{proof}[Proof of Proposition \ref{prop:nn_apprx_spec_barron}]
    Let $k \in \mathbb{N}$. Denote $\hat{f}(\xi) = |\hat{f}(\xi)| e^{\ii \theta(\xi)}$. Let $\mu$ be a probability measure on $\mathbb{R}^d$ with density
    \begin{align}
        \dfrac{|\hat{f}(\xi)| (1 + |\xi|)^k}{\Vert f \Vert_{\mathcal{B}^{k}(\mathbb{R}^d)}}. \nonumber
    \end{align}
    Since $f \in \mathcal{B}^0(\mathbb{R}^d)$, $\hat{f} \in L^1(\mathbb{R}^d)$. The real-valued function $f$ can be written as
    \begin{align}
        f(x) & = \int_{\mathbb{R}^d} \hat{f}(\xi) e^{\ii \xi \cdot x} \dd \xi \nonumber \\
        & = \int_{\mathbb{R}^d} |\hat{f}(\xi)| e^{\ii(\xi \cdot x + \theta(\xi))} \dd \xi \nonumber \\
        & = \int_{\mathbb{R}^d} |\hat{f}(\xi)| \left(1 + |\xi| \right)^{k} \left(1 + |\xi| \right)^{-k}e^{\ii(\xi \cdot x + \theta(\xi))} \dd \xi \nonumber \\
        & = \int_{\mathbb{R}^d} |\hat{f}(\xi)| \left(1 + |\xi| \right)^{k} \left(1 + |\xi| \right)^{-k} \cos\left(\xi \cdot x + \theta(\xi )\right) \dd \xi \nonumber \\
        & = \Vert f \Vert_{\mathcal{B}^{k}(\mathbb{R}^d)} \int_{\mathbb{R}^d} \dfrac{|\hat{f}(\xi)| \left(1 + |\xi| \right)^{k}}{\Vert f \Vert_{\mathcal{B}^{k}(\mathbb{R}^d)}} \left(1 + |\xi| \right)^{-k} \cos\left(\xi\cdot x + \theta(\xi  )\right) \dd \xi \nonumber \\
        & = \Vert f \Vert_{\mathcal{B}^{k}(\mathbb{R}^d)} \mathbb{E}_{\xi\sim\mu} \left[ \left(1 + |\xi| \right)^{-k}\cos\left(\xi \cdot x + \theta(\xi) \right)\right]. \nonumber
    \end{align}
    By Lebesgue's dominated convergence theorem, $f \in C(\mathbb{R}^d)$. If $k \geq 1$, then for each $1 \leq i \leq d$, the partial derivative 
    \begin{align}
        \partial_i f(x) = - \Vert f \Vert_{\mathcal{B}^{0}(\mathbb{R}^d)} \mathbb{E}_{\xi \sim \mu} \left[ \left(1 + |\xi| \right)^{-k} \xi_i \sin(\xi \cdot x + \theta(\xi))\right] \nonumber
    \end{align}
    exists and is continuous on $\mathbb{R}^d$ by Lebesgue's dominated convergence theorem. Hence $f \in C^1(\mathbb{R}^d)$. If $k \geq 2$, then the second-order partial derivative
    \begin{align}
        \partial_{ij} f(x) = - \Vert f \Vert_{\mathcal{B}^{k}(\mathbb{R}^d)} \mathbb{E}_{\xi\sim\mu} \left[ \left(1 + |\xi| \right)^{-k} \xi_i \xi_j \cos\left(\xi \cdot x + \theta(\xi) \right)\right] \nonumber
    \end{align}
    exists and is continuous on $\mathbb{R}^d$ by the dominated convergence theorem again. Thus $f \in C^2(\mathbb{R}^d)$. By repeating the argument, one can show that if $f \in \mathcal{B}^{k}(\mathbb{R}^d)$, then $f \in C^{k}(\mathbb{R}^d)$. 
    
    Let $\xi_1, ..., \xi_n$ be independent and identically distributed samples from $\mu$, and 
    \begin{align}
        f_n(x) := \dfrac{1}{n} \sum_{l=1}^{n} a_l \cos\left(w_l \cdot x +b_l \right), \nonumber
    \end{align}
    where 
    \begin{align}
        a_l = \Vert f \Vert_{\mathcal{B}^{k}(\mathbb{R}^d)}\left(1 + |\xi_l| \right)^{-k} , \hspace{1em} w_l = \xi_l, \hspace{1em} b_l = \theta(\xi_l). \nonumber
    \end{align}
    Then 
    \begin{align}
        \mathbb{E}_{\mu^{\otimes n}} \left[ \Vert f_n - f \Vert^2_{L^2(K)} \right] & = \mathbb{E}_{\mu^{\otimes n}}\left[ \int_{K} |f_n(x) - f(x)|^2 \dd x\right] \nonumber \\
        = \Vert f \Vert^2_{\mathcal{B}^{k}(\mathbb{R}^d)} & \int_{K} \mathrm{Var}_{\mu^{\otimes n}} \left[ \dfrac{1 }{n} \sum_{l=1 }^{n}\left(1 + |\xi_l| \right)^{-k} \cos\left( \xi_l \cdot x + \theta(\xi_l )\right) \right] \dd x \nonumber \\
        = \dfrac{\Vert f \Vert^2_{\mathcal{B}^{k}(\mathbb{R}^d)}}{n} & \int_{K} \mathrm{Var}_{\xi \sim \mu} \left[ \left(1 + |\xi| \right)^{-k} \cos\left( \xi \cdot x + \theta(\xi)\right)\right] \dd x \nonumber \\
        \leq \dfrac{\Vert f \Vert^2_{\mathcal{B}^{k}(\mathbb{R}^d)}}{n} & \int_{K} \mathbb{E}_{\xi \sim \mu} \left[ \left(1 + |\xi| \right)^{-2k} \cos^2 \left( \xi \cdot x + \theta(\xi)\right)\right] \dd x \nonumber \\
        \leq \dfrac{\Vert f \Vert^2_{\mathcal{B}^{k}(\mathbb{R}^d)}}{n} & \int_{K} \mathbb{E}_{\xi \sim \mu} \left[ \left(1 + |\xi| \right)^{-2k} \right] \dd x. \nonumber
    \end{align}
    Since $(1 + |\xi|)^{-2k} \leq 1$, 
    \begin{align}
        \mathbb{E}_{\mu^{\otimes n}} \left[ \Vert f_n - f \Vert^2_{L^2(K)} \right] \leq  \dfrac{|K| \Vert f \Vert^2_{\mathcal{B}^{k}(\mathbb{R}^d)}}{n}. \nonumber
    \end{align}
    If $k \geq 1$, the partial derivative $\partial_i f$ exists for $1 \leq i \leq d$. Therefore, 
    \begin{align}
        \mathbb{E}_{\mu^{\otimes n}} \left[ \sum_{i=1}^{d} \Vert \partial_i f_n - \partial_i f \Vert^2_{L^2(K)} \right] & = \mathbb{E}_{\mu^{\otimes n}} \left[\sum_{i=1}^{d} \int_{K} |\partial_i f_n(x) - \partial_i f(x)|^2 \dd x \right] \nonumber \\
        = \Vert f \Vert^2_{\mathcal{B}^{k}(\mathbb{R}^d)} & \sum_{i=1}^{d} \int_{K} \mathrm{Var}_{\mu^{\otimes n}} \left[ \dfrac{1 }{n} \sum_{l=1 }^{n}\left(1 + |\xi_l| \right)^{-k} \xi_{l,i} \sin\left( \xi_l \cdot x + \theta(\xi_l )\right) \right] \dd x \nonumber \\
        = \dfrac{\Vert f \Vert^2_{\mathcal{B}^{k}(\mathbb{R}^d)}}{n} & \sum_{i=1}^{d} \int_{K} \mathrm{Var}_{\xi \sim \mu} \left[ \left(1 + |\xi| \right)^{-k} \xi \sin\left( \xi \cdot x + \theta(\xi)\right)\right] \dd x  \nonumber \\
        \leq \dfrac{\Vert f \Vert^2_{\mathcal{B}^{k}(\mathbb{R}^d)}}{n} & \int_{K} \mathbb{E}_{\xi \sim \mu} \left[ \left(1 + |\xi| \right)^{-2k} |\xi|^2 \sin^2 \left( \xi \cdot x + \theta(\xi)\right)\right] \dd x \nonumber \\
        \leq \dfrac{\Vert f \Vert^2_{\mathcal{B}^{k}(\mathbb{R}^d)}}{n} & \int_{K} \mathbb{E}_{\xi \sim \mu} \left[ \left(1 + |\xi| \right)^{-2k} |\xi|^2 \right] \dd x. \nonumber
    \end{align}
    Since $(1 +|\xi|)^{-2k}(1 + |\xi|^2) \leq 1$ when $k \geq 1$, we obtain
    \begin{align}
        \mathbb{E}_{\mu^{\otimes n}} \left[ \Vert f_n - f \Vert^2_{H^1(K)} \right] \leq \dfrac{\Vert f \Vert^2_{\mathcal{B}^{k}(\mathbb{R}^d)}}{n} & \int_{K} \mathbb{E}_{\xi \sim \mu} \left[ \left(1 + |\xi| \right)^{-2k} (1 + |\xi|^2) \right] \dd x \leq  \dfrac{|K| \Vert f \Vert^2_{\mathcal{B}^{k}(\mathbb{R}^d)}}{n}. \nonumber
    \end{align}
    If $k \geq 2$, the second-order partial derivative $\partial_{ij} f$ exists for any $1 \leq i, j \leq d$. Similarly, we obtain 
    \begin{align}
        \mathbb{E}_{\mu^{\otimes n}} \left[ \sum_{i,j=1}^{d} \Vert \partial_{ij} f_n - \partial_{ij} f \Vert^2_{L^2(K)} \right] 
        \leq \dfrac{\Vert f \Vert^2_{\mathcal{B}^{k}(\mathbb{R}^d)}}{n} & \int_{K} \mathbb{E}_{\xi \sim \mu} \left[ \left(1 + |\xi| \right)^{-2k} |\xi|^4 \right] \dd x. \nonumber
    \end{align}
    Since $\left(1 + |\xi| \right)^{-2k} \left(1 + |\xi|^2 + |\xi|^4 \right) \leq 1 $ when $k \geq 2$, it follows that
    \begin{align}
        \mathbb{E}_{\mu^{\otimes n}} \left[ \Vert f_n - f \Vert^2_{H^2(K)} \right] & \leq \dfrac{\Vert f \Vert^2_{\mathcal{B}^{k}(\mathbb{R}^d)}}{n} \int_{K} \mathbb{E}_{\xi \sim \mu} \left[ \left(1 + |\xi| \right)^{-2k} \left(1 + |\xi|^2 + |\xi|^4 \right) \right] \dd x \nonumber\\
        & \leq  \dfrac{|K| \Vert f \Vert^2_{\mathcal{B}^{k}(\mathbb{R}^d)}}{n}. \nonumber
    \end{align}
    By applying the argument in the same way $k+1$ times, one can conclude that
    \begin{equation*}
        \mathbb{E}_{\mu^{\otimes n}} \left[ \Vert f_n - f \Vert^2_{H^k(K)} \right] \leq \dfrac{|K| \Vert f \Vert^2_{\mathcal{B}^{k}(\mathbb{R}^d)}}{n}. \qedhere
    \end{equation*}
\end{proof}

\begin{proof}[Proof of Proposition \ref{prop:barron_alg}]
    (1) If $f, g \in \mathcal{B}^{s}(\mathbb{R}^d)$, then
    \begin{align}
        \Vert fg \Vert_{\mathcal{B}^s(\mathbb{R}^d)} & \leq \int_{\mathbb{R}^d} \int_{\mathbb{R}^d} |\hat{f}(\xi)| |\hat{g}(\eta - \xi)| (1 + |\eta|)^{s} \dd \xi \dd \eta \nonumber \\
        & \leq \int_{\mathbb{R}^d} \int_{\mathbb{R}^d} |\hat{f}(\xi)| |\hat{g}(\eta - \xi)| (1 + |\xi|)^s (1 + |\eta - \xi|)^s \dd\xi \dd\eta. \nonumber \\
        & =  \Vert f \Vert_{\mathcal{B}^s(\mathbb{R}^d)} \Vert g \Vert_{\mathcal{B}^s(\mathbb{R}^d)}. \nonumber
    \end{align}
    (2) If $f \in \mathcal{B}^{s+1}(\mathbb{R}^d)$, then 
    \begin{align}
        \Vert \partial_{x_i} f \Vert_{\mathcal{B}^s(\mathbb{R}^d)} & = \int_{\mathbb{R}^d} \left( 1 + |\xi | \right)^{s} \left| \widehat{\partial_{x_i} f} (\xi)\right| \dd \xi \nonumber \\
        & = \int_{\mathbb{R}^d} \left( 1 + |\xi | \right)^{s} \left| \ii \xi_i \widehat{f} (\xi)\right| \dd \xi \nonumber \\
        & =  \int_{\mathbb{R}^d} \left( 1 + |\xi | \right)^{s} \left| \xi_i  \right| \cdot \left| \widehat{f} (\xi)\right| \dd \xi \nonumber \\
        & \leq \int_{\mathbb{R}^d} \left( 1 + |\xi | \right)^{s+1} \left| \widehat{f} (\xi)\right| \dd \xi = \Vert f \Vert_{\mathcal{B}^{s+1}}(\mathbb{R}^d). \nonumber \qedhere
    \end{align}
\end{proof}

\begin{proof}[Proof of Lemma \ref{lmm:constant_c_R1}]
    For any control $u$, we have 
    \begin{align}
        \langle u, u \rangle_{R} = \sum_{i=1}^{m} \sum_{j=1}^{m} u_i R_{ij} u_j. \nonumber
    \end{align}
    Since $R$ is a positive-definite matrix, $\max_{1 \leq i,j \leq m} R_{ij} > 0$. If $u \in \mathcal{B}^{s}(\mathbb{R}^d)$, then 
    \begin{align}
        \Vert \langle u, u\rangle_R \Vert_{\mathcal{B}^{s}(\mathbb{R}^d)} & \leq \max_{1 \leq i,j \leq m} R_{ij} \sum_{i=1}^{m} \sum_{j=1}^{m} \Vert u_i u_j \Vert_{\mathcal{B}^{s}(\mathbb{R}^d)} \nonumber \\
        & \leq \max_{1 \leq i,j \leq m} R_{ij} \sum_{i=1}^{m} \sum_{j=1}^{m} \Vert u_i \Vert_{\mathcal{B}^{s}(\mathbb{R}^d)} \Vert u_j \Vert_{\mathcal{B}^{s}(\mathbb{R}^d)} \hspace{1em} \text{(Proposition \ref{prop:barron_alg} (1))}\nonumber \\
        & \leq \max_{1 \leq i,j \leq m} R_{ij} \Vert u \Vert_{\mathcal{B}^{s}(\mathbb{R}^d)}^2. \nonumber
    \end{align}
    By choosing $C_{R,1} = \max_{1\leq i,j \leq m} R_{ij}$, the claim is proved. 
\end{proof}

\begin{proof}[Proof of Lemma \ref{lmm:constant_c_R2}]
    For each $1 \leq i \leq m$, 
    \begin{align}
        \left( R^{-1} g^T \nabla V \right)_i = \sum_{j=1}^{m} \left(R^{-1} \right)_{ij} \left( g^T \nabla V \right)_j, \nonumber
    \end{align}
    and for each $1 \leq j \leq m$, 
    \begin{align}
        \left( g^T \nabla V \right)_j = \sum_{k=1}^{d} g_{kj} \partial_{k} V. \nonumber
    \end{align}
    If $V \in \mathcal{B}^{s+1}(\mathbb{R}^d)$, then 
    \begin{align}
        \Vert \left( g^T \nabla V \right)_j \Vert_{\mathcal{B}^s(\mathbb{R}^d)} & \leq \sum_{k=1}^{d} \Vert g_{kj} \Vert_{\mathcal{B}^s(\mathbb{R}^d)} \Vert \partial_k V \Vert_{\mathcal{B}^s(\mathbb{R}^d)} \nonumber \\
        & \leq \sum_{k=1}^{d} \Vert g_{kj} \Vert_{\mathcal{B}^s(\mathbb{R}^d)} \Vert V \Vert_{\mathcal{B}^{s+1}(\mathbb{R}^d)}. \hspace{1em} \text{(Proposition \ref{prop:barron_alg} (2))}\nonumber
    \end{align}
    Then
    \begin{align}
        \Vert \left( R^{-1} g^T \nabla V \right)_i \Vert_{\mathcal{B}^s(\mathbb{R}^d)} & \leq \max_{1 \leq j \leq m} \left| \left( R^{-1} \right)_{ij}\right| \sum_{j=1}^{m} \Vert \left( g^T \nabla V \right)_j \Vert_{\mathcal{B}^s(\mathbb{R}^d)} \nonumber \\
        & \leq \max_{1 \leq j \leq m} \left| \left( R^{-1} \right)_{ij}\right| \sum_{j=1}^{m} \sum_{k=1}^{d} \Vert g_{kj} \Vert_{\mathcal{B}^s(\mathbb{R}^d)} \Vert V \Vert_{\mathcal{B}^{s+1}(\mathbb{R}^d)} \nonumber \\
        & = \max_{1 \leq j \leq m} \left| \left( R^{-1} \right)_{ij}\right| \Vert g \Vert_{\mathcal{B}^{s}(\mathbb{R}^d)} \Vert V \Vert_{\mathcal{B}^{s+1}(\mathbb{R}^d)}, \nonumber
    \end{align}
    and so
    \begin{align}
        \Vert  R^{-1} g^T \nabla V \Vert_{\mathcal{B}^s(\mathbb{R}^d)} \leq \sum_{i=1}^{m} \Bigl( \max_{1 \leq j \leq m} \left| \left( R^{-1} \right)_{ij}\right| \Bigr) \Vert g \Vert_{\mathcal{B}^{s}(\mathbb{R}^d)} \Vert V \Vert_{\mathcal{B}^{s+1}(\mathbb{R}^d)}. \nonumber
    \end{align}
    For the claim to be established, it suffices to choose 
    \begin{equation*}
        C_{R,2} := \dfrac{1 }{2 } \sum_{i=1}^{m} \Bigl( \max_{1 \leq j \leq m} \left| \left( R^{-1} \right)_{ij}\right| \Bigr). \qedhere
    \end{equation*}
\end{proof}

\bibliographystyle{plain}
\bibliography{refs}

\begin{thebibliography}{10}

\bibitem{barron1993universal}
Andrew~R Barron.
\newblock Universal approximation bounds for superpositions of a sigmoidal function.
\newblock {\em IEEE Transactions on Information theory}, 39(3):930--945, 1993.

\bibitem{beck2020overview}
Christian Beck, Martin Hutzenthaler, Arnulf Jentzen, and Benno Kuckuck.
\newblock An overview on deep learning-based approximation methods for partial differential equations.
\newblock {\em arXiv preprint arXiv:2012.12348}, 2020.

\bibitem{chen2021representation}
Ziang Chen, Jianfeng Lu, and Yulong Lu.
\newblock On the representation of solutions to elliptic pdes in barron spaces.
\newblock {\em Advances in neural information processing systems}, 34:6454--6465, 2021.

\bibitem{chen2023regularity}
Ziang Chen, Jianfeng Lu, Yulong Lu, and Shengxuan Zhou.
\newblock A regularity theory for static schr{\"o}dinger equations on d in spectral barron spaces.
\newblock {\em SIAM Journal on Mathematical Analysis}, 55(1):557--570, 2023.

\bibitem{weinan2021algorithms}
Weinan E, Jiequn Han, and Arnulf Jentzen.
\newblock Algorithms for solving high dimensional pdes: from nonlinear monte carlo to machine learning.
\newblock {\em Nonlinearity}, 35(1):278, 2021.

\bibitem{wojtowytsch2022representation}
Weinan E and Stephan Wojtowytsch.
\newblock Representation formulas and pointwise properties for barron functions.
\newblock {\em Calculus of Variations and Partial Differential Equations}, 61(2):1--37, 2022.

\bibitem{weinan2022some}
Weinan E and Stephan Wojtowytsch.
\newblock Some observations on high-dimensional partial differential equations with barron data.
\newblock In {\em Mathematical and Scientific Machine Learning}, pages 253--269. PMLR, 2022.

\bibitem{gilbarg1977elliptic}
David Gilbarg and Neil~S Trudinger.
\newblock {\em Elliptic partial differential equations of second order}, volume 224.
\newblock Springer, 1977.

\bibitem{grohs2021deep}
Philipp Grohs and Lukas Herrmann.
\newblock Deep neural network approximation for high-dimensional parabolic hamilton-jacobi-bellman equations.
\newblock {\em arXiv preprint arXiv:2103.05744}, 2021.

\bibitem{han2016deep}
Jiequn Han and Weinan E.
\newblock Deep learning approximation for stochastic control problems.
\newblock {\em arXiv preprint arXiv:1611.07422}, 2016.

\bibitem{han2018solving}
Jiequn Han, Arnulf Jentzen, and Weinan E.
\newblock Solving high-dimensional partial differential equations using deep learning.
\newblock {\em Proceedings of the National Academy of Sciences}, 115(34):8505--8510, 2018.

\bibitem{han2019solving}
Jiequn Han, Linfeng Zhang, and Weinan E.
\newblock Solving many-electron schr{\"o}dinger equation using deep neural networks.
\newblock {\em Journal of Computational Physics}, 399:108929, 2019.

\bibitem{hanche2010kolmogorov}
Harald Hanche-Olsen and Helge Holden.
\newblock The kolmogorov--riesz compactness theorem.
\newblock {\em Expositiones Mathematicae}, 28(4):385--394, 2010.

\bibitem{hermann2020deep}
Jan Hermann, Zeno Sch{\"a}tzle, and Frank No{\'e}.
\newblock Deep-neural-network solution of the electronic schr{\"o}dinger equation.
\newblock {\em Nature Chemistry}, 12(10):891--897, 2020.

\bibitem{klebaner2012introduction}
Fima~C Klebaner.
\newblock {\em Introduction to stochastic calculus with applications}.
\newblock World Scientific Publishing Company, 2012.

\bibitem{li2024neural}
Xingjian Li, Deepanshu Verma, and Lars Ruthotto.
\newblock A neural network approach for stochastic optimal control.
\newblock {\em SIAM Journal on Scientific Computing}, 46(5):C535--C556, 2024.

\bibitem{lu2022priori}
Jianfeng Lu and Yulong Lu.
\newblock A priori generalization error analysis of two-layer neural networks for solving high dimensional schr{\"o}dinger eigenvalue problems.
\newblock {\em Communications of the American Mathematical Society}, 2(1):1--21, 2022.

\bibitem{lu2021priori}
Yulong Lu, Jianfeng Lu, and Min Wang.
\newblock A priori generalization analysis of the deep ritz method for solving high dimensional elliptic partial differential equations.
\newblock In {\em Conference on learning theory}, pages 3196--3241. PMLR, 2021.

\bibitem{ma2022barron}
Chao Ma, Lei Wu, and Weinan E.
\newblock The barron space and the flow-induced function spaces for neural network models.
\newblock {\em Constructive Approximation}, 55(1):369--406, 2022.

\bibitem{marwah2023neural}
Tanya Marwah, Zachary~Chase Lipton, Jianfeng Lu, and Andrej Risteski.
\newblock Neural network approximations of pdes beyond linearity: A representational perspective.
\newblock In {\em International Conference on Machine Learning}, pages 24139--24172. PMLR, 2023.

\bibitem{Marwah2021ParametricCB}
Tanya Marwah, Zachary~Chase Lipton, and Andrej Risteski.
\newblock Parametric complexity bounds for approximating pdes with neural networks.
\newblock {\em ArXiv}, abs/2103.02138, 2021.

\bibitem{nakamura2021adaptive}
Tenavi Nakamura-Zimmerer, Qi~Gong, and Wei Kang.
\newblock Adaptive deep learning for high-dimensional hamilton-jacobi-bellman equations.
\newblock {\em SIAM Journal on Scientific Computing}, 43(2):A1221--A1247, 2021.

\bibitem{nusken2021solving}
Nikolas N{\"u}sken and Lorenz Richter.
\newblock Solving high-dimensional hamilton-jacobi-bellman pdes using neural networks: perspectives from the theory of controlled diffusions and measures on path space.
\newblock {\em Partial differential equations and applications}, 2(4):48, 2021.

\bibitem{pfau2020ab}
David Pfau, James~S Spencer, Alexander~GDG Matthews, and W~Matthew~C Foulkes.
\newblock Ab initio solution of the many-electron schr{\"o}dinger equation with deep neural networks.
\newblock {\em Physical review research}, 2(3):033429, 2020.

\bibitem{siegel2020approximation}
Jonathan~W Siegel and Jinchao Xu.
\newblock Approximation rates for neural networks with general activation functions.
\newblock {\em Neural Networks}, 128:313--321, 2020.

\bibitem{siegel2022high}
Jonathan~W Siegel and Jinchao Xu.
\newblock High-order approximation rates for shallow neural networks with cosine and reluk activation functions.
\newblock {\em Applied and Computational Harmonic Analysis}, 58:1--26, 2022.

\bibitem{siegel2023characterization}
Jonathan~W Siegel and Jinchao Xu.
\newblock Characterization of the variation spaces corresponding to shallow neural networks.
\newblock {\em Constructive Approximation}, 57(3):1109--1132, 2023.

\bibitem{siegel2024sharp}
Jonathan~W Siegel and Jinchao Xu.
\newblock Sharp bounds on the approximation rates, metric entropy, and n-widths of shallow neural networks.
\newblock {\em Foundations of Computational Mathematics}, 24(2):481--537, 2024.

\bibitem{xu2020finite}
Jinchao Xu.
\newblock The finite neuron method and convergence analysis.
\newblock {\em arXiv preprint arXiv:2010.01458}, 2020.

\bibitem{yong2012stochastic}
Jiongmin Yong and Xun~Yu Zhou.
\newblock {\em Stochastic controls: Hamiltonian systems and HJB equations}, volume~43.
\newblock Springer Science \& Business Media, 2012.

\bibitem{yserentant2025regularity}
Harry Yserentant.
\newblock The regularity of electronic wave functions in barron spaces.
\newblock {\em arXiv preprint arXiv:2502.17950}, 2025.

\bibitem{zhou2021actor}
Mo~Zhou, Jiequn Han, and Jianfeng Lu.
\newblock Actor-critic method for high dimensional static hamilton--jacobi--bellman partial differential equations based on neural networks.
\newblock {\em SIAM Journal on Scientific Computing}, 43(6):A4043--A4066, 2021.

\end{thebibliography}

\end{document}